\documentclass[11pt]{article}

\usepackage[margin=1in]{geometry}
\usepackage{amsmath,amssymb,amsthm,mathtools}
\usepackage{enumitem}
\usepackage[colorlinks=true,linkcolor=blue,citecolor=blue,urlcolor=blue]{hyperref}

\newtheorem{theorem}{Theorem}
\newtheorem{lemma}[theorem]{Lemma}
\newtheorem{proposition}[theorem]{Proposition}
\newtheorem{corollary}[theorem]{Corollary}
\theoremstyle{definition}
\newtheorem{definition}[theorem]{Definition}
\theoremstyle{remark}
\newtheorem{remark}[theorem]{Remark}

\newcommand{\F}{\mathbb{F}}
\newcommand{\Z}{\mathbb{Z}}
\newcommand{\Tr}{\mathrm{Tr}}
\newcommand{\ord}{\mathrm{ord}}
\newcommand{\abs}[1]{\left|#1\right|}

\newcommand{\la}{\langle}
\newcommand{\ra}{\rangle}

\title{On discrepancy estimates for pseudorandom vectors constructed by the elliptic curve congruential generator}
\author{{$^{1,2}$}Ziran Liu, {$^1$}Chung Pang Mok\\[0.5em]
$^1$Shanghai Institute for Mathematics and Interdisciplinary Sciences (SIMIS)\\Shanghai 200433, China\\
$^2$Research Institute of Intelligent Complex Systems, Fudan University,\\ Shanghai 200433, China\\
[0.4em]
\texttt{zliu@simis.cn, cpmok@simis.cn}
}

\date{}

\begin{document}
\maketitle

\begin{abstract}
This paper studies the problem of discrepancy estimates for pseudorandom vectors constructed by the elliptic curve congruential generator, particularly in the non-translational case.
Two families of results are obtained.
First, in a full-coset regime characterized by a \emph{relative maximal period condition} (RMPC) on an induced one-dimensional linear congruential generator, one proves bounds of type $q^{1/2}/t$ for the discrepancy $D$, the serial discrepancy $D_s$, and, under the corresponding derived RMPC, the non-overlapping discrepancy $\widetilde D_s$.
Second, in the general sub-period regime, one reduces bounds for $D$, $D_s$, and $\widetilde D_s$ to estimation of Fourier $\ell^1$ masses of admissible index sets attached to one-dimensional linear congruential generators.
This isolates the arithmetic bottleneck for further improvement.
\end{abstract}

\tableofcontents

\section{Introduction}

Let $E/\F_q$ be an elliptic curve over a finite field of cardinality $q$ with identity element $O\in E(\F_q)$; here as usual $E(\F_q)$ is the group of points of $E/\F_q$ with coordinates in $\F_q$.
Fix a non-zero integer $e$ and a point $Q\in E(\F_q)$.
Consider the affine iteration using the addition law on $E/\F_q$ (here in general for any integer $e$ then $[e]$ is the multiplication by $e$ map on $E$):
\begin{equation}\label{eq:affine-iter}
P_{n+1}=[e]P_n+Q,\qquad n\ge 0,
\end{equation}
starting from a point $P_0\in E(\F_q)$ (the seed); thus we have the points $P_n \in E(\F_q)$ for $n \geq0$ (recall that by the Hasse bound we have $|\# E(\F_q) - (q+1)| \leq 2 q^{1/2}$). This could be regarded as the elliptic curve version of the linear congruential generator, or the elliptic curve congruential generator for short (despite this name, it should be noted though that it is a highly non-linear generator, due to the fact that the addition law of $E$ is given by non-linear rational functions of the coordinates). To ease notations we assume throughout the paper that $P_0$ lies in a purely periodic orbit, and $t$ denotes the period, i.e. the least positive integer such that $P_t=P_0$.

\bigskip
In \cite{Mok2023} an algorithm is given for the construction of pseudorandom vectors in a unit hypercube. The constructions are obtained from the output map when it is being applied to the orbit $\{P_n\}_{ n \geq 0}$ (we recall the definitions in section two; see also \cite{MZ} where Monte Carlo numerical experiments are carried out by using these pseudorandom vectors). In {\it loc. cit.} we established, in the translation case $e=1$ (in this paper we call this as case A), explicit upper bounds for the discrepancy $D$ (that gives quantitative measure of uniformity of distribution), the serial $s$-discrepancy $D_s$ and its non-overlapping version $\widetilde D_s$ (that give quantitative measure of statistical independence, here $2 \leq s \leq t$), by reducing these bounds to bounds for additive character sums over (cyclic) subgroups of $E(\F_q)$, and then applying square-root cancellation for such sums. These discrepancies bounds in particular justify the name {\it pseudorandom} when $t \gg q^{1/2 + \epsilon}$ for $\epsilon > 0$ (at least in the translational case $e=1$).
In this paper we focus on the case where $e \neq 1$. This presents much more serious difficulties, though it was observed in {\it loc. cit.} that if the period is maximal, i.e. $t=\#E(\F_q)$, then the same estimate for the discrepancy $D$ remains valid by reducing the situation to a maximal translation orbit.
However the estimation of the serial $s$-discrepancy $D_s$, and the non-overlapping $s$-discrepancy $\widetilde D_s$ remain open when $e\neq 1$ (even in the case $t = \#E(\F_q)$).

\bigskip
The main results proved in this paper are of two kinds:
\begin{itemize}[leftmargin=*]
\item[1.] In a \emph{full-coset regime} (Case B), the orbit does not necessarily cover $E(\F_q)$, but it covers a full coset of a cyclic subgroup canonically attached to $(e,P_0,Q)$.
In this regime one recovers bounds of type $q^{1/2}/t$ as in \cite{Mok2023}.
\item[2.] In the \emph{general sub-period regime} (Case C), the orbit is a proper subset of that coset.
In this regime the discrepancy problem is reduced to the estimation of Fourier $\ell^1$ mass of an admissible index set associated with a one-dimensional linear congruential generator.
\end{itemize}

The paper is organized as follows.
Section~2 recalls some of the materials from \cite{Mok2023} including the definitions of the output map, the notion of discrepancies, relation with Walsh functions, and exponential sums over subgroups of $E(\F_q)$.
Section~3 treats the full-coset regime and proves bounds of size $q^{1/2}/t$ under the relative maximal period condition.
Section~4 treats the general sub-period regime and reduces all three discrepancy problems to the estimation of Fourier masses.
Section~5 analyzes the arithmetic structure of the index sequence and derives the strengthened results in the no mixed local factors regime (see Corollary~\ref{cor:no-mixed} of Section 5.1 for the definition).

\section{R\'esum\'e from \cite{Mok2023}}

\subsection{The digital output map}
As usual $p$ is the characteristic of the finite field $\F_q$. Write
\[
q=p^{ar}
\]
with integers $a,r\ge 1$, and fix the tower of field extensions:
\[
\F_q/\F_{p^a}/\F_p.
\]
(in Monte Carlo applications one fixes a choice of $a,r$ and $p$, and then put $q=p^{ar}$). Choose an $\F_{p^a}$-basis $\{\lambda_1,\dots,\lambda_r\}$ of $\F_q$ with the corresponding dual basis $\{\lambda^{\prime}_1,\dots,\lambda^{\prime}_r\}$ with respect to $\Tr_{\F_q/\F_{p^a}}$, and an $\F_p$-basis $\{\kappa_1,\dots,\kappa_a\}$ of $\F_{p^a}$, with the corresponding dual basis $\{\kappa^{\prime}_1,\dots,\kappa^{\prime}_a \}$ with respect to $\Tr_{\F_{p^a}/\F_p}$.
Following section 2.3 of \cite{Mok2023}, define for any $\eta\in \F_q$ and $1\le j\le r$:
\[
\phi_j(\eta):=\sum_{i=1}^a \frac{\Tr_{\F_q/\F_p}(\eta\lambda^{\prime}_j\kappa^{\prime}_i)}{p^i}\in [0,1).
\]
(here we fix $\{0,1,\cdots,p-1\} \subset \Z$ to be the set of representatives of elements of $\F_p$). We remark that in the notation of \cite{Mok2023} we have $\phi_j(\eta) = \Phi(\langle \eta \rangle_j)$ for any $\eta\in \F_q$ and $1\le j\le r$. 

\bigskip
Let $x,y$ be the usual affine Weierstrass coordinates for $E/\F_q$ (with respect to a choice of affine Weierstrass equation for $E/\F_q$). Define 
\[
G: E(\F_q)\setminus\{O\} \rightarrow [0,1)^{2r}
\]
by the following: for $P\in E(\F_q)\setminus\{O\}$, put
\begin{equation}\label{eq:Phi}
G(P):=\big(\phi_1(x(P)),\dots,\phi_r(x(P)),\phi_1(y(P)),\dots,\phi_r(y(P))\big)\in [0,1)^{2r}.
\end{equation}
This is the the output map $G$ defined in \cite{Mok2023} (with respect to the above choices): for each $P \in E(\F_q) \setminus\{O\}  $, the elements $x(P)$ and $y(P)$ of $\F_q$ are expanded relative to the chosen field bases and then converted into points of $[0,1)$ by base-$p$ expansions. Thus we obtain the point $G(P)$ in the unit hypercube $[0,1)^{2r}$. The output map $G$ is easily seen to be injective.

\bigskip
In \cite{Mok2023} one also assigns an output value to $O$, namely $G(O)=(1,\dots,1)$ (in fact assigning any point in $[0,1]^{2r} \setminus [0,1)^{2r}$ as a value for $G(O)$ works equally well). This gives the (injective) output map $G:E(\F_q) \rightarrow [0,1]^{2r}$

\subsection{Notation for admissible output collections}
For the orbit $\{P_n\}_{ n \ge 0}$ of \eqref{eq:affine-iter}, which we assume as in the Introduction that it is purely periodic with $t$ being the period: $P_{t} = P_0$, the admissible output collection $\mathcal{P} \subset [0,1)^{2r}$ is:
\[
\mathcal P:=\big\{G(P_n)\big\}_{\substack{0\le n < t\\ P_n\neq O}}.
\]
Its cardinality is denoted as \(N^\ast\).

\bigskip
For integers \(2\le s\le t\), the admissible serial \(s\)-tuple collection, denoted by \(\mathcal P^{(s)} \subset [0,1)^{2rs}\) , is given by:
\[
\big\{(G(P_n),\dots,G(P_{n+s-1} ))\big\}_{\substack{0\le n < t\\ P_n,\cdots,P_{n+s-1}\neq O}}
\]
with the block $(G(P_n),\dots,G(P_{n+s-1} ))$ being regarded as a point of $[0,1)^{2rs}$. The cardinality of \(\mathcal P^{(s)}\) is denoted as \(N_s^\ast\).

\bigskip

Finally for $2 \le s \le t$ the admissible non-overlapping \(s\)-tuple collection, denoted by \(\widetilde{\mathcal P}^{(s)} \subset [0,1)^{2rs}\), is given by:
\[
\big\{(G(P_{ns}),G(P_{ns+1})\dots,G(P_{ns+s-1} ))\big\}_{\substack{0\le n < t/{\rm{gcd}}(s,t) \\ P_{ns},P_{ns+1}\cdots,P_{ns+s-1}\neq O}}
\]
with the block $(G(P_{ns}),G(P_{ns+1})\dots,G(P_{ns+s-1} )) $ again being regarded as a point of $[0,1)^{2rs}$. The cardinality of \(\widetilde{\mathcal P}^{(s)}\) is denoted as \(\widetilde N_s^\ast\).

\bigskip
If the orbit $\{P_n\}_{ n \geq 0}$ avoids \(O\), then \(N^\ast=N^{\ast}_s=t\), and $\widetilde N_s^\ast  = t/{\rm{gcd}}(s,t)$. If the orbit contains \(O\), the admissible counts are obtained by deleting exactly those points or blocks in which \(O\) appears; in this case we have $N^{\ast}=t-1$, $N^{\ast}_s =t-s$, and $\widetilde N^{\ast}_s = (t-s)/\gcd(s,t)$.

\bigskip

Throughout the paper, any bound on discrepancy (the definition is recalled in section 2.3 below) that involves:
\[
\mathcal{P}, \quad \mathcal{P}^{(s)}, \quad \widetilde{\mathcal P}^{(s)}
\]
is understood under the natural assumption that the corresponding admissible point collection is nonempty.

\subsection{Point collections and discrepancy}
Firstly consider in general any finite collection of points $\mathcal{P}$ in $[0,1)^d$ (here $d$ is any positive integer).
Here below we consider axis-parallel $d$-dimensional box of the form:
\[
B=\prod_{j=1}^d [a_j,b_j)\subset [0,1)^d,
\]
The \emph{extreme discrepancy} or just discrepancy of $\mathcal P \subset [0,1)^d$ is
\[
D(\mathcal P):=\sup_B\left|\frac{ \# (\mathcal{P} \cap B)   }{\#\mathcal{P} }-\mathrm{vol}(B)\right|,
\]
where the supremum runs over all axis-parallel $d$-dimensional boxes $B$ in $[0,1)^d$ (it is regarded as undefined if $\mathcal{P}$ is empty). We have $0 \le D(\mathcal P) \le 1$ and it is a quantitative measure of the uniformity of distribution of $\mathcal{P}$ with respect to the usual Lebesgue measure on $[0,1)^d$ (the closer is $D(\mathcal{P})$ to zero, the more uniform is the distribution).

\bigskip
\noindent Now in general consider a finite sequence of $t$ \emph{distinct} points:
\[
\{v_n\}_{0 \le n < t},\quad v_n\in [0,1]^d
\]
in the unit hypercube $[0,1]^d$ of dimension $d$. Denote by $\mathcal{P} \subset[0,1)^d$ the admissible collection given by:
\[
\mathcal{P} = \{v_0,\cdots,v_{t-1} \} \cap [0,1)^d
\]
The discrepancy of the sequence $\{v_n\}_{0 \le n < t}$ is then defined to be the discrepancy of $\mathcal{P} \subset [0,1)^d $.

\bigskip
\noindent For $2\le s \le t$ the admissible \emph{serial $s$-tuple collection} is
\[
\mathcal P^{(s)}:=\big\{(v_n,\dots,v_{n+s-1})\big\}_{\substack{0\le n < t \\ v_n,\cdots,v_{n+s-1} \in [0,1)^d }}
\]
where the indices are taken modulo $t$, and the block $(v_n,\dots,v_{n+s-1})$ being understood as a point of $[0,1)^{ds}$;
the corresponding serial $s$-discrepancy of of the sequence $\{v_n\}_{0 \le n < t}$, is defined to be the discrepancy of $\mathcal P^{(s)} \subset [0,1)^{ds}$.

\bigskip

\noindent We similarly define the admissible \emph{non-overlapping} $s$-tuple collection as:
\[
\widetilde{\mathcal P}^{(s)}:=\big\{(v_{ns},v_{ns+1},\dots,v_{ns+s-1})\big\}_{\substack{0\le n < t/\gcd(s,t) \\ v_{ns},v_{ns+1},\cdots,v_{ns+s-1} \in [0,1)^d }}
\]
again with indices taken modulo $t$, and the block $(v_{ns},v_{ns+1}\dots,v_{ns+s-1})$ being understood as a point of $[0,1)^{ds}$.
The corresponding non-overlapping $s$-discrepancy of of the sequence $\{v_n\}_{0 \le n < t}$, is defined to be the discrepancy of $\widetilde{\mathcal P}^{(s)} \subset [0,1)^{ds}$.

\bigskip

Given the finite sequence $\{v_n\}_{0 \le n < t}$ the discrepancy is a quantitative measure of the uniformity of distribution (with respect to the usual Lebesgue measure of $[0,1)^d$), while the serial $s$-discrepancy (or its non-overlapping version) is a quantitative measure of statistical independence of $s$ successive terms of the sequence.  

\subsection{Walsh function and Erd\"os-Turan-Koksma type inequality}
We recall the base-$p$ Walsh function system. Firstly some notations. For any non-negative integer $k$ we write:
\[
k = \sum_{i=1}^{\infty} k(i) p^{i-1}, \qquad k(i) \in \{0,1,\cdots,p-1\}
\]
be the unique expansion of $k$ in base $p$ (all but finitely many of the $k(i)$'s are equal to zero). Next every $\xi \in [0,1)$ has a unique base $p$ expansion:
\[
\xi=\sum_{i=1}^{\infty} \frac{\xi(i)}{p^i}, \qquad \xi(i) \in \{0,1,\cdots,p-1\}
\]
with the condition that $\xi(i) \neq p-1$ for infinitely many $i$. Then define:
\[
w_k(\xi)  =\exp\Big(\frac{2\pi i}{p} \sum_{i=1}^{\infty} k(i) \xi(i) \Big)
\]
We refer to $k$ as the Walsh frequency.

\bigskip

In dimension $d \geq 1$, for ${\bf k} = (k^{(1)},\cdots,k^{(d)})$ where $k^{(1)},\cdots,k^{(d)}$ are non-negative integers, and $\xi=(\xi^{(1)} \cdots,\xi^{(d)})$, where $\xi^{(1)},\cdots,\xi^{(d)} \in [0,1)$, define:
\[
w_{{\bf k}}(\xi) =  w_{k^{(1)}}(\xi^{(1)}) \cdots w_{k^{(d)}}(\xi^{(d)})
\]
We refer to ${\bf k}$ as the vector Walsh frequency, or for short, again just as Walsh frequency. 

\bigskip

For $\mathcal{P} \subset [0,1)^d$ a (non-empty) finite set, put for Walsh frequency ${\bf k}$:
\[
S(w_{{\bf k}} ,\mathcal{P}) = \frac{1}{\# \mathcal{P}} \sum_{\xi \in \mathcal{P}} w_{{\bf k}}(\xi)
\]

\bigskip

Now fix a positive integer $a$ and we make the assumption that for all $\xi \in \mathcal{P}$, all the coordinates of $p^a \cdot \xi$ are integers. Note that this condition is satisfied with $\mathcal{P}$ being the admissible output collection, or the admissible serial (or non-overlapping) $s$-tuple collection as in section 2.2 (with the same value of $a$ in section 2.1). Now with respect to this fixed value of $a$ and the dimension $d$, we let $\Delta_d$ to be the set of all ${\bf k} = (k^{(1)},\cdots,k^{(d)})$ such that $0 \le k^{(1)},\cdots, k^{(d)} <p^a$. Finally put $\Delta_d^{\ast} = \Delta_d \setminus \{(0,\cdots,0)\}$.

\bigskip

We now state the Erd\"os-Turn-Koksma type inequality with respect to the Walsh function system as established by Hellekalek \cite{H}:

\begin{lemma} (Corollary 4 of \cite{H}) \label{lem:Walsh-discrepancy}
Let $\mathcal P \subset [0,1)^d$ be a finite point collection lying on the grid $\{0,1/p^a,\dots,(p^a-1)/p^a\}^d$.
Assume that for every nonzero Walsh frequency ${\bf k} \in \Delta_d^{\ast}$ one has:
\[
|S(w_{\bf k},\mathcal P)|\le B,
\]
for some real number $B$.
Then there is some absolute constant $C_0$ such that:
\begin{equation}\label{eq:Hellekalek}
D(\mathcal P)\le 1-\Big(1-\frac{1}{p^a}\Big)^d + B\,\big(C_0\ln(p^a)+1\big)^d,
\end{equation}
\end{lemma}

\noindent In \cite{H} a value of $C_0$ is given as 2.43, or one can take $C_0=1.78$ by \cite[Lemma~2.1]{Wang2024}. Note also that the term:
\[
1-\Big(1-\frac{1}{p^a}\Big)^d
\]
is as a discretization error term.

\subsection{Evaluation of Walsh function on the image of the output map}
We now consider the evaluation of the Walsh function on the image of the output map $G$; thus the dimension $d$ is now equal to $2r$. The value of $a$ is as in section 2.1. 

\bigskip
\noindent Firstly let
\[
\psi_1(z):=\exp\Big(\frac{2\pi i}{p}\Tr_{\F_q/\F_p}(z)\Big),\qquad z\in \F_q,
\]
be the standard non-trivial additive character of $\F_q$. Any other non-trivial additive character of $\F_q$ is then given by $\psi_a$ for some $a \in \F_q^{\times}$, where $\psi_a(z) = \psi_1(a \cdot z)$ for $z \in \F_q$.

\begin{lemma} \label{lem:Walsh-to-char}
For each nonzero Walsh frequency ${\bf k }\in\Delta_{2r}^{\ast}$ there exists elements $\eta_{\bf k},\widetilde\eta_{\bf k}\in\F_q$, not both zero, such that for all $P\in E(\F_q)\setminus\{O\}$, we have:
\[
w_{\bf k}(G(P))=\psi_1\big(\eta_{\bf k} \cdot  x(P)+\widetilde\eta_{\bf k} \cdot y(P)\big).
\]
Consequently, for any subset $S\subset E(\F_q)$, and the point collection
\[
\mathcal P=\{G(P)\}_{P\in S,\ P\neq O},
\]
one has
\begin{equation}\label{eq:Walsh-to-char-sum}
|\mathcal P|\cdot S(w_{\bf k},\mathcal P)
=
\sum_{P\in S,\ P\neq O}\psi_1\big(\eta_{\bf k} \cdot x(P)+\widetilde\eta_{\bf k} \cdot  y(P)\big)
\end{equation}

\end{lemma}

\begin{proof}
This is a standard computation, see \cite[(5.2)--(5.3)]{Mok2023} or \cite[(2.2)--(2.3)]{Wang2024}. Explicitly with ${\bf k} = (k^{(1)},\cdots,k^{(2r)})$ we have:
\[
\eta_{{\bf k}} = \sum_{j=1}^r \sum_{i=1}^a k^{(j)}(i) \lambda_j^{\prime} \kappa_i^{\prime} 
\]
and
\[
\widetilde \eta_{{\bf k}} = \sum_{j=1}^r \sum_{i=1}^a k^{(r+j)}(i) \lambda_j^{\prime} \kappa_i^{\prime} 
\]
\end{proof}

\subsection{Twisted subgroup character sums}
We need the following square-root cancellation estimate of Kohel-Shparlinski \cite{KS}. Firstly $\F_q(E)$ is the function field of $E$ over $\F_q$, and $\overline{\F}_q(E)$ is the function field of $E$ over $\overline{\F}_q$ (the algebraic closure of $\F_q$). We say that an element $f \in \F_q(E) $ is {\it non-degenerate}, if $f \neq g^p - g $ for any $g \in \overline{\F}_q(E)$ ($f$ is thus necessarily non-constant). 

\begin{lemma} (Corollary 1 of \cite{KS}) \label{lem:KS}
Let $H\subset E(\F_q)$ be a subgroup, and $\omega$ be a group character of $H$. 
For $\psi$ a nontrivial additive character of $\F_q$ and a non-degenerate $f\in \F_q(E)$, we have:
\begin{equation}\label{eq:KS}
\abs{\sum_{\substack{P\in H\\ f(P)\neq\infty}}\omega(P)\psi(f(P))}
\le 2\deg(f)\,q^{1/2}.
\end{equation}
Moreover, if the polar divisor of $f$ is supported on a single prime divisor of $E$, then
\begin{equation}\label{eq:KS-single-pole}
\abs{\sum_{\substack{P\in H\\ f(P)\neq\infty}}\omega(P)\psi(f(P))}
\le (1+\deg(f))\,q^{1/2}.
\end{equation}
\end{lemma}

\subsection{Auxiliary facts on separable multiplication maps and poles}

The following two lemmas are standard and are used repeatedly in the proofs. These concern the geometric properties of $E$, i.e. of $E$ over $\overline{\F}_q$.

\begin{lemma} \label{lem:pullback-separable}
Let $n\in \Z$.
Then $[n]:E\to E$ is an isogeny of degree $n^2$; it is separable if $\gcd(n,p)=1$.
Moreover:
\begin{enumerate}[label=(\roman*),leftmargin=*]
\item for every $A\in E(\overline{\F}_q)$, the translation map $\tau_A(P):=P+A$ is an automorphism of $E$ over $\overline{\F}_q$ (as a genus one curve), so in particular is of degree $1$;
\item for every nonconstant $f\in \overline{\F}_q(E)$,
\[
\deg(f\circ \tau_A)=\deg(f),\qquad \deg(f\circ [n])=n^2\deg(f),
\]
and hence
\[
\deg(f\circ (\tau_A\circ [n]))=n^2\deg(f);
\]
\item if $\gcd(n,p)=1$ and $f$ has a pole of order $m$ at $Q\in E(\overline{\F}_q)$, then $f\circ (\tau_A\circ [n])$ has poles of order $m$ at the set $[n]^{-1}(Q-A)$.
In particular, since $[n]$ is separable, this set has cardinality equal to $n^2$ over $\overline{\F}_q$.
\end{enumerate}
\end{lemma}

\begin{proof}
We work with $E$ over $\overline{\F}_q$. It is immediate that translations are automorphisms of $E$ (as a genus one curve). It is also standard fact that the map $[n]$ is an isogeny of degree $n^2$ and is separable whenever $\gcd(n,p)=1$. The degree identities follow from standard facts on morphisms of curves and pullback of rational functions.

For part (iii), if $f$ has a pole of order $m$ at $Q$, then $f\circ\tau_A$ has a pole of order $m$ at $Q-A$.
Pulling back by the separable map $[n]$ preserves the pole order at each point of the fibre and yields exactly $n^2$ distinct preimages.
\end{proof}

\begin{lemma} \label{lem:AS-pole-criterion}
Let $f \in \overline{\F}_q(E)$.
If $f $ has a pole whose order is not divisible by $p$, then $f \neq h^p-h$ for any $h\in \overline{\F}_q(E)$.
\end{lemma}

\begin{proof}
We again work with $E$ over $\overline{\F}_q$. If $h$ has a pole of order $m$ at some point of $E(\overline{\F}_q)$, then $h^p-h$ has a pole of order $pm$ there, since the term $h^p$ dominates $h$ at the pole.
Therefore every pole order of a function of the form $h^p-h$ is divisible by $p$. The claim follows.
\end{proof}

\section{The full-coset regime (Case B)}

Throughout this section assume $|e| \ge 2$ (indeed when $e=-1$ then the period of the orbit is at most two and so this case could be ignored). As in the Introduction we assume that the orbit $\{P_n\}_{n \geq 0}$ is purely periodic and we let $t$ be the period.

\subsection{Orbit reduction to a cyclic coset}

\begin{lemma}[Orbit reduction to a cyclic coset]\label{lem:orbit-coset}
Let $\{P_n\}_{n \geq 0}$ satisfy \eqref{eq:affine-iter}. Define
\[
\beta_n:=\frac{e^n-1}{e-1}\in\Z\qquad (n\ge 0; \ \beta_0=0),
\qquad
R:=P_1-P_0=(e-1)P_0+Q\in E(\F_q).
\]
Then for all $n\ge 0$,
\begin{equation}\label{eq:Pn-betaR}
P_n=P_0+[\beta_n]R.
\end{equation}
In particular, the full orbit $\{P_n\}_{n \geq 0}$ is contained in the coset $P_0+\la R\ra$, where $\la R \ra$ is the cyclic subgroup of $E(\F_q)$ generated by $R$.
\end{lemma}

\begin{proof}
By induction on $n$ one first obtains
\[
P_n=[e^n]P_0+[\beta_n]Q,
\]
using the recursion $\beta_{n+1}=e\beta_n+1$.
Subtracting $P_0$ gives
\[
P_n-P_0=[e^n-1]P_0+[\beta_n]Q=[\beta_n]\big((e-1)P_0+Q\big)=[\beta_n]R,
\]
which is \eqref{eq:Pn-betaR}.
\end{proof}

\begin{lemma}\label{lem:gcd-em}
With $R$ as above put $m=\ord(R)$, the order of $R$ in $E(\F_q)$. Then:
\[
\gcd(e,m)=1.
\]
\end{lemma}

\begin{proof}
Let $d:=\gcd(e,m)$. If $d>1$, then the defining recursion $\beta_{n+1}=e\beta_n+1$ gives, modulo $d$,
\[
\beta_{n+1}\equiv e\beta_n+1\equiv 1 \pmod d
\qquad (n\ge 0),
\]
because $d\mid e$. Hence
\[
\beta_n\equiv 1\pmod d
\qquad\text{for all }n\ge 1,
\]
while $\beta_0\equiv 0\pmod d$.
On the other hand as $t$ is the period of the orbit we have $P_t=P_0$, so Lemma~\ref{lem:orbit-coset} gives
\[
[\beta_t]R=O.
\]
Since $\ord(R)=m$, it follows that $m\mid \beta_t$, hence $d\mid \beta_t$.
This contradicts $\beta_t\equiv 1\pmod d$.
Therefore $d=1$.
\end{proof}

Thus let
\[
m:=\ord(R),\qquad H:=\la R\ra\subset E(\F_q).
\]
Reducing \eqref{eq:Pn-betaR} modulo $H$, one sees that the period $t$ of the orbit $\{P_n\}_{n \geq 0}$ is equal to the period of the sequence $\{\beta_n  \bmod{m} \}_{n \geq 0}$.
Moreover,
\begin{equation}\label{eq:index-LCG}
\beta_{n+1}\equiv e\beta_n+1\pmod m,\qquad \beta_0\equiv 0 \pmod m.
\end{equation}
Thus the triple $(e,P_0,Q)$ that defines the elliptic curve congruential generator \eqref{eq:affine-iter} then induces a linear congruential generator (LCG) on $\Z/m\Z$.

\subsection{The relative maximal period condition}
\begin{definition}[RMPC]\label{def:RMPC}
The triple $(e,P_0,Q)$ is said to satisfy the \emph{relative maximal period condition} if the LCG \eqref{eq:index-LCG} has full period $m=\ord(R)$, that is,
\[
\{\beta_0,\dots,\beta_{t-1} \bmod{m} \}=\Z/m\Z.
\]
Equivalently, $t=m$, i.e. the orbit point set is the full coset $P_0+\la R\ra$. 

\end{definition}

\begin{lemma}[Hull--Dobell criterion for RMPC]\label{lem:Hull-Dobell}
The LCG \eqref{eq:index-LCG} has full period $m$ if and only if:
\begin{enumerate}[label=(\roman*),leftmargin=*]
\item for every prime $\ell\mid m$, one has $\ell\mid (e-1)$;
\item if $4\mid m$, then $4\mid (e-1)$.
\end{enumerate}
\end{lemma}

\begin{proof}
This is the classical Hull--Dobell criterion for the LCG $x_{n+1}\equiv ex_n+1\pmod m$ to have maximal period.
\end{proof}

\bigskip
As a simple example: take $Q=O$; let $M= \ord(P_0)$ and in the following $\ell$ denotes a prime dividing $M$; put $e=1+ \prod_{\ell |M} \ell$ if  $8$ does not divide $M$, and $e=1+ 2 \prod_{\ell |M} \ell$ if 8 divides $M$. Then $e-1$ divides $M$, and $R=P_1 -P_0 = [e-1] P_0$, hence $m=\ord(R)= M/(e-1)$. The Hull-Dobell criterion for RMPC is then seen to be satisfied. 

\subsection{Discrepancy bound under RMPC}

\begin{theorem}[Discrepancy bound in the full-coset regime]\label{thm:CaseB-D}
Assume RMPC and so $t=m=\ord(R)$.
Define the admissible point collection as in section 2.1:
\[
\mathcal P := \{ G(P_n)\}_{\substack{0\le n < t\\ P_n\neq O}} \subset [0,1)^{2r},
\qquad
N^{\ast}:=|\mathcal P|.
\]
(recall that $N^{\ast} =t$ if the orbit $\{P_n\}_{n \ge 0}$ avoids $O$, and is equal to $t-1$ otherwise).

\bigskip
Assume $p\ge 5$.
Then
\begin{equation}\label{eq:CaseB-D}
D(\mathcal P)
\le 1-\Big(1-\frac{1}{p^a}\Big)^{2r}
+\frac{4q^{1/2}}{N^{\ast}}\big(C_0\ln(p^a)+1\big)^{2r}.
\end{equation}
\end{theorem}

\begin{proof}
Under RMPC, Lemma~\ref{lem:orbit-coset} gives:
\[
\{P_0,\dots,P_{t-1}\}=\{P_0+[u]R: u\in \Z/t\Z\}.
\]

Consider any nonzero Walsh frequency ${\bf k} \in\Delta_{2r}^{\ast}$.
By Lemma~\ref{lem:Walsh-to-char},
\[
N^{\ast}S(w_{\bf k},\mathcal P)
=
\sum_{\substack{u=0\\ P_0+[u]R\neq O}}^{t-1}
\psi_1\big(\eta_{\bf k} \cdot x(P_0+[u]R)+\widetilde\eta_{\bf k} \cdot y(P_0+[u]R)\big).
\]
with $\eta_{\bf k},\widetilde\eta_{\bf k} \in \F_q $, not noth equal to zero.

\bigskip
Define
\[
f_{{\bf k},P_0}:=\eta_{\bf k} \cdot  x  \circ \tau_{P_0}+\widetilde\eta_{\bf k} \cdot y \circ \tau_{P_0}\in \F_q(E).
\]
The polar divisor of $f_{{\bf k},P_0}$ is supported at the single point $-P_0$, with the order of pole being equal to either $2$ or $3$ (it is equal to $2$ exactly when $\widetilde\eta_{\bf k}$ is zero); hence $\deg(f_{k,P_0})\le 3$.
Since $p\ge 5$, it follows in particular that the order of pole is not divisible by $p$; therefore Lemma~\ref{lem:AS-pole-criterion} shows that $f_{{\bf k},P_0}$ is non-degenerate. We then have with $H=\la R \ra$:
\begin{eqnarray*}
N^{\ast}S(w_{\bf k},\mathcal P)
&=&   \sum_{\substack{u=0\\ P_0+[u]R\neq O}}^{t-1} \psi_1(f_{{\bf k},P_0})([u]R)   \\
&=&\sum_{\substack{P\in H\\ f_{{\bf k},P_0}(P)\neq\infty}} \psi_1(f_{{\bf k},P_0})(P)
\end{eqnarray*}

Now applying Lemma~\ref{lem:KS} in the single-pole form \eqref{eq:KS-single-pole} yields
\[
\abs{\sum_{\substack{P\in H\\ f_{{\bf k},P_0}(P)\neq\infty}}\psi_1(f_{{\bf k},P_0}(P))}
\le (1+\deg(f_{{\bf k},P_0}))q^{1/2}\le 4q^{1/2}.
\]
Hence $|S(w_{\bf k},\mathcal P)|\le 4q^{1/2}/N^\ast$ for every nonzero Walsh frequency ${\bf k} \in \Delta_{2r}^{\ast}$.
Applying Lemma~\ref{lem:Walsh-discrepancy} with $d=2r$ proves \eqref{eq:CaseB-D}.
\end{proof}

\subsection{Serial discrepancy bound under RMPC}

\begin{theorem}[Serial discrepancy in the full-coset regime]\label{thm:CaseB-Ds}
Assume RMPC, $\gcd(e,p)=1$, and $p\ge 5$.
Fix $2\le s \le t$.
Let $\mathcal P^{(s)} \subset [0,1)^{2rs}$ be the admissible serial $s$-tuple collection attached to the orbit $\{P_n\}_{n \ge 0}$, and let $N_s^{\ast}:=|\mathcal P^{(s)}|$.

\bigskip

\noindent Then
\begin{equation}\label{eq:CaseB-Ds}
D(\mathcal P^{(s)})
\le 1-\Big(1-\frac{1}{p^a}\Big)^{2rs}
+\frac{6q^{1/2}\big(\sum_{\iota =0}^{s-1}e^{2 \iota}\big)+s}{N_s^{\ast}}\big(C_0\ln(p^a)+1\big)^{2rs}.
\end{equation}
\end{theorem}

\begin{proof}
Again under RMPC we have $t=m=\ord(R)$.
Using the identity
\[
\beta_{n+\iota}=\beta_{\iota}+e^{\iota}\beta_n
\]
in $\Z$, one obtains
\[
P_{n+\iota}=P_{\iota}+[e^{\iota} u]R,
\qquad
P_{\iota}=P_0+[\beta_{\iota}]R,
\qquad
u:=\beta_n\bmod t.
\]

\bigskip
Define:
\[
U_s:=\{u\in\Z/t\Z: P_{\iota}+[e^{\iota} u]R\neq O\ \text{for all }0\le \iota\le s-1\}
\]
then the admissible serial $s$-tuple point collection $\mathcal P^{(s)}$ is the image of $U_s$ under the injective map:
\[
u \in \Z/t \Z \longmapsto
\big(G(P_0+[u]R),G(P_1+[eu]R),\dots,G(P_{s-1}+[e^{s-1}u]R)\big) \in [0,1]^{2rs}.
\]

In particular $|U_s|=| \mathcal{P}^{(s)}|=N_s^{\ast}$, that we recall is equal to $t$ (resp. $t-s$) if the orbit $\{P_n\}_{n \geq 0}$ avoids $O$ (resp. otherwise).

\bigskip
Now given a nonzero Walsh frequency ${\bf k}\in\Delta_{2rs}^{\ast}$, we then obtain as in Lemma~\ref{lem:Walsh-to-char} (see for example \cite[(5.6)--(5.7)]{Mok2023}), coefficients $(\eta_{{\bf k},\iota},\widetilde\eta_{{\bf k},\iota})\in\F_q^2$ for $0 \le \iota \le s-1$, not all zero, such that
\[
N_s^{\ast} S(w_{\bf k},\mathcal P^{(s)})=
\sum_{u\in U_s}\psi_1\big(g_{\bf k}([u]R)\big),
\]
explicitly,
\[
\eta_{{\bf k},\iota} = \sum_{j=1}^r \sum_{i=1}^a k^{(2 r \iota + j)}(i) \lambda_j^{\prime} \kappa_i^{\prime}   ,
\]
\[
\widetilde \eta_{{\bf k},\iota} = \sum_{j=1}^r \sum_{i=1}^a k^{(2 r \iota+r+j)}(i) \lambda_j^{\prime} \kappa_i^{\prime}   ,
\]
and
\[
g_{\bf k}:=
\sum_{\iota=0}^{s-1}
\Big(\eta_{{\bf k},\iota} \cdot x \circ \tau_{P_{\iota}} \circ [e^{\iota}] +\widetilde\eta_{{\bf k},\iota} \cdot y \circ \tau_{P_{\iota}} \circ [e^{\iota}] \Big)
\in \F_q(E).
\]

\bigskip

For each $0 \le \iota \le s-1$ the element:
\[
\eta_{{\bf k},\iota} \cdot x   +\widetilde\eta_{{\bf k},\iota} \cdot y 
\in \F_q(E).
\]
is either identically equal to zero (namely when $\eta_{{\bf k},\iota} $ and $\widetilde\eta_{{\bf k},\iota}$ are both zero), or is non-constant with degree equal to $2$ or $3$, with a single pole supported at $O$ (of order equal to $2$ or $3$). Hence the $\iota$-th summand (in the sum that defines $g_{{\bf k}}$):
\[
\eta_{{\bf k},\iota} \cdot x \circ \tau_{P_{\iota}} \circ [e^{\iota}] +\widetilde\eta_{{\bf k},\iota} \cdot y \circ \tau_{P_{\iota}} \circ [e^{\iota}] 
\in \F_q(E)
\]
is either identically equal to zero, or is non-constant with degree equal to $2e^{2 \iota}$ or $3e^{2 \iota}$. In particular we obtain:
\begin{equation}\label{eq:deg-gk}
\deg(g_{\bf k})\le 3\sum_{\iota=0}^{s-1}e^{2 \iota}.
\end{equation}

We next show that $g_{{\bf k}}$ is non-degenerate for each ${\bf k} \in \Delta_{2rs}^{\ast}$.

\bigskip

Let $0 \le \iota_0 \le s-1$ be maximal with $(\eta_{k,\iota_0},\widetilde\eta_{k,\iota_0})\neq (0,0)$.
By Lemma~\ref{lem:pullback-separable} (here we use the condition that $\gcd(e,p)=1$), the $\iota_0$-th summand:
\[
\eta_{{\bf k},\iota_0} \cdot x \circ \tau_{P_{\iota_0}} \circ [e^{\iota_0}] +\widetilde\eta_{{\bf k},\iota_0} \cdot y \circ \tau_{P_{\iota_0}} \circ [e^{\iota_0}]  \in \F_q(E)
\]
has pole set $[e^{\iota_0}]^{-1}(-P_{\iota_0})$ over $\overline{\F}_q$, with cardinality $e^{2\iota_0}$, and the order of these poles is either all equal to $2$, or all equal to $3$.

\bigskip
Now for each $\iota<\iota_0$, the cardinality of the pole set of the $\iota$-th summand has is at most $e^{2 \iota}$.
Hence the union of all pole sets for $\iota<\iota_0$ has cardinality at most
\[
\sum_{\iota=0}^{\iota_0-1}e^{2\iota}<e^{2\iota_0}
\qquad (|e|\ge 2).
\]
Therefore there exists a point $P^{\ast}$ which is a pole of the $\iota_0$-th summand and not a pole of any smaller $\iota$-th summand.
At $P^{\ast}$ the $\iota_0$-th summand has pole order $2$ or $3$. Thus $g_{\bf k}$ itself has a pole of order $2$ or $3$ at $P^{\ast}$.
Since $p\ge 5$, this pole order is not divisible by $p$, and Lemma~\ref{lem:AS-pole-criterion} implies that $g_{\bf k}$ is non-degenerate.

\bigskip

Finally for every ${\bf k} \in \Delta_{2rs}^{\ast}$ let
\[
U_{\bf k}:=\{u\in\Z/t\Z: g_{\bf k}([u]R)\neq \infty\}.
\]
Then $U_s\subseteq U_{\bf k}$ for every ${\bf k} \in \Delta_{2rs}^{\ast}$, and in any case:
\[
|U_{\bf k}\setminus U_s| \le   | \Z/t\Z\setminus U_s |\le s.
\]
Since each summand $\psi_1(g_{\bf k}([u]R))$ has absolute value $1$ on $U_{\bf k}$, one obtains
\[
\abs{\sum_{u\in U_s}\psi_1(g_{\bf k}([u]R))}
\le
\abs{\sum_{u\in U_{\bf k}}\psi_1(g_{\bf k}([u]R))}+s.
\]

\noindent Now with $H=\la R \ra$ we have:
\[
\sum_{u\in U_{\bf k}}\psi_1(g_{\bf k}([u]R)) = \sum_{ \substack{P\in H\\ g_{{\bf k}}(P)\neq\infty}}\psi_1(g_{\bf k}(  P ))
\]
so Lemma~\ref{lem:KS} applied to the sum on the RHS (which is applicable as we have shown that $g_{{\bf k}}$ is non-degenerate), together with \eqref{eq:deg-gk}, give:
\[
\abs{\sum_{u\in U_{\bf k}}\psi_1(g_{\bf k}([u]R))}
\le 2\deg(g_{\bf k})q^{1/2}
\le 6q^{1/2}\sum_{\iota=0}^{s-1}e^{2 \iota}.
\]
Thus to conclude for every ${\bf k} \in \Delta_{2rs}^{\ast}$ we obtain:
\[
\abs{N_s^{\ast}S(w_{\bf k},\mathcal P^{(s)})}
\le 6q^{1/2}\sum_{\iota=0}^{s-1}e^{2\iota}+s.
\]
Applying Lemma~\ref{lem:Walsh-discrepancy} with $d=2rs$ proves \eqref{eq:CaseB-Ds}.
\end{proof}

\begin{remark}
If the orbit $\{P_n\}_{n \ge 0}$ avoids $O$, then the boundary correction term $+s$ in \eqref{eq:CaseB-Ds} is not needed.
\end{remark}

\subsection{Non-overlapping discrepancy bound under a derived RMPC}

\begin{definition}[Derived RMPC]\label{def:derived-RMPC}
Fix $s\ge 2$ and define
\[
R_s:=P_s-P_0, \quad
\gamma_n = \frac{e^{ns}-1}{e^s-1}
\]
Let $m_s:=\ord(R_s)$.
The \emph{derived RMPC of step $s$} is the requirement that the LCG
\[
\gamma_{n+1}\equiv e^s\gamma_n+1\pmod{m_s},
\qquad \gamma_0\equiv 0 \pmod{m_s},
\]
has full period $m_s$.
\end{definition}

\bigskip
Just as in Lemma~\ref{lem:Hull-Dobell}, by using the classical Hull-Dobell criterion, we see that the derived RMPC of step $s$ is satisfied, if and only if:
\begin{enumerate}[label=(\roman*),leftmargin=*]
\item for every prime $\ell\mid m_s$, one has $\ell\mid (e^s-1)$;
\item if $4\mid m_s$, then $4\mid (e^s-1)$.
\end{enumerate}

\bigskip

In particular as $m_s | m$ (because $R_s=P_s-P_0 = [\beta_s] R \in \la R \ra$), we see that RMPC implies derived RMPC of step $s$ for any $s \ge 2$. 

\bigskip

\begin{theorem}[Non-overlapping discrepancy in the full-coset regime]\label{thm:CaseB-Dtilde}
Assume $|e|\ge 2$, $\gcd(e,p)=1$, and $p\ge 5$.
Fix $s\ge 2$ and assume the derived RMPC of step $s$.
Let $\widetilde{\mathcal P}^{(s)}$ be the admissible non-overlapping $s$-tuple collection, and let $\widetilde N_s^{\ast}:=|\widetilde{\mathcal P}^{(s)}|$.
Then
\begin{equation}\label{eq:CaseB-Dtilde}
D(\widetilde{\mathcal P}^{(s)})
\le 1-\Big(1-\frac{1}{p^a}\Big)^{2rs}
+\frac{6q^{1/2}\big(\sum_{\iota=0}^{s-1}e^{2\iota}\big)+s/\gcd(s,t)}{\widetilde N_s^{\ast}}\big(C_0\ln(p^a)+1\big)^{2rs}.
\end{equation}
\end{theorem}

\begin{proof}
Firstly, the period of the sequence $\{P_{ns}\}_{n \ge 0}$ is equal to $L_s:=t/\gcd(s,t)$, and by direct computation one has
\[
P_{ns}=P_0+[\gamma_n]R_s,
\]
from which is follows firstly that that $L_s$ is equal to the period of the LCG $\{\gamma_n \bmod{m_s}\}_{n \geq 0}$; secondly, the derived RMPC of step $s$ is equivalent to the condition that the sequence $\{P_{ns}\}_{n \geq 0}$ traverses the full coset $P_0 +  \la R_s \ra$, i.e. that:
\[
L_s=m_s.
\]

\noindent More generally one again obtain by direct computation that for $0\le \iota \le s-1$ and $n\ge 0$:
\[
P_{ns+\iota}=P_{\iota}+[e^{\iota}\gamma_n]R_s.
\]

\noindent Thus under the derived RMPC of step $s$, define:
\[
\widetilde U_s:=\{u\in\Z/L_s\Z: P_{\iota}+[e^{\iota} u]R_s\neq O\text{ for all }0\le \iota\le s-1\}
\]
then the admissible non-overlapping $s$-tuple collection $\widetilde{\mathcal P}^{(s)}$ is the image of $\widetilde U_s$ under the injective map: 
\[
u \in \Z/ L_s \Z \longmapsto
\big(G(P_0+[u]R_s),G(P_1+[eu]R_s),\dots,G(P_{s-1}+[e^{s-1}u]R_s)\big) \in [0,1]^{2rs}.
\]
Thus $|\widetilde U_s|=|\widetilde{\mathcal P}^{(s)}|= \widetilde N_s^{\ast}$, that we recall is equal to $L_s=t/\gcd(s,t)$ if the orbit $\{P_n\}_{n \geq 0}$ avoids $O$, and is equal to $(t-s)/\gcd(s,t)$ otherwise. In particular that
\[
|\Z/L_s\Z\setminus \widetilde U_s|\le s/\gcd(s,t).
\]

The rest of the proof is then similar to the proof of Theorem~\ref{thm:CaseB-Ds}. Fix ${\bf k} \in\Delta_{2rs}^{\ast}$.
Then we have by similar considerations: 
\[
\widetilde N_s^{\ast}S(w_{\bf k},\widetilde{\mathcal P}^{(s)})
=
\sum_{u\in \widetilde U_s}\psi_1\big(g_{\bf k}([u]R_s)\big),
\]
where $g_{{\bf k}} \in \F_q(E)$ is the same as in {\it loc. cit.}:
\[
g_{\bf k}:=
\sum_{\iota=0}^{s-1}
\Big(\eta_{{\bf k},\iota} \cdot x \circ \tau_{P_{\iota}} \circ [e^{\iota}] +\widetilde\eta_{{\bf k},\iota} \cdot y \circ \tau_{P_{\iota}} \circ [e^{\iota}] \Big)
\]
and recall that we have established that
\[
\deg(g_{\bf k})\le 3\sum_{\iota=0}^{s-1}e^{2\iota}
\]
and that $g_{{\bf k}}$ is non-degenerate (under the same assumptions that $p \geq 5$ and $\gcd(e,p)=1$).

\bigskip

Let
\[
\widetilde U_{\bf k}:=\{u\in\Z/L_s\Z:\  g_{\bf k}([u]R_s)\neq \infty\}.
\]
Then $\widetilde U_s\subseteq \widetilde U_{\bf k}$, and so in any case that:
\[
|\widetilde U_{\bf k}\setminus \widetilde U_s|\le |\Z/L_s \Z  \setminus \widetilde U_s| \le s/\gcd(s,t).
\]
Therefore
\[
\abs{\sum_{u\in\widetilde U_s}\psi_1( g_{\bf k}([u]R_s))}
\le
\abs{\sum_{u\in\widetilde U_k}\psi_1( g_{\bf k}([u]R_s))}+s/\gcd(s,t).
\]
Applying Lemma~\ref{lem:KS} with $H$ being the cyclic subgroup $\la R_s\ra$ yields
\[
\abs{\sum_{u\in\widetilde U_{\bf k}}\psi_1( g_{\bf k}([u]R_s))}
\le 2\deg( g_{\bf k})q^{1/2}
\le 6q^{1/2}\sum_{\iota=0}^{s-1}e^{2\iota}.
\]
Therefore
\[
\abs{\widetilde N_s^{\ast}S(w_{\bf k},\widetilde{\mathcal P}^{(s)})}
\le 6q^{1/2}\sum_{\iota=0}^{s-1}e^{2 \iota}+s/\gcd(s,t).
\]
Finally applying Lemma~\ref{lem:Walsh-discrepancy} in dimension $2rs$ proves \eqref{eq:CaseB-Dtilde}.
\end{proof}

\begin{remark}
If the orbit $\{P_n\}_{n \ge 0}$ avoids $O$, then the boundary correction term $+s/\gcd(s,t)$ in \eqref{eq:CaseB-Dtilde} is not needed.
\end{remark}

\section{The general sub-period regime (Case C)}

We retain the notation
\[
R=(e-1)P_0+Q,\qquad m=\ord(R),\qquad H=\la R\ra,
\]
\[
\beta_n:=\frac{e^n-1}{e-1}\in\Z\qquad (n\ge 0; \ \beta_0=0)
\]
and we write
\[
B=\{\beta_0,\dots,\beta_{t-1} \bmod{m} \}\subset \Z/m\Z
\]
for the cycle associated with the LCG \eqref{eq:index-LCG} (recall that the period $t$ of $\{P_n\}_{n \ge 0}$ coincides with the period of the LCG \eqref{eq:index-LCG}; in particular $|B|=t$). Unlike Case B (namely RMPC), in this section we do {\bf not} require $t=m$ (nor do we require the derived RMPC); thus apriori we could not directly apply Lemma~\ref{lem:KS}. To deal with this we use Fourier techniques. 

\subsection{Definition of the Fourier $\ell^1$ mass}

\begin{definition}[Fourier transform and Fourier $\ell^1$ mass]\label{def:L1M}
Let $M\ge 1$ be a positive integer and let $A\subset \Z/M\Z$ be an arbitrary subset.
Define the Fourier transform of $\mathbf 1_A$ (the characteristic function of $A$ as a subset of $\Z/M \Z$) as:
\[
\widehat{\mathbf 1_A}^{(M)}(j):=\sum_{u\in A}e^{-2\pi iju/M},\qquad 0\le j<M,
\]
and the Fourier $\ell^1$ mass of $A$:
\[
\mathcal L_M(A):=\frac{1}{M}\sum_{j=0}^{M-1}\abs{\widehat{\mathbf 1_A}^{(M)}(j)}.
\]
\end{definition}

\bigskip

We have the elementary but useful:

\begin{lemma}\label{lem:L1-sqrt}
Let $M\ge 1$ and let $A\subset \Z/M\Z$ be an arbitrary subset.
Then
\[
\mathcal L_M(A)\le \sqrt{|A|}.
\]
\end{lemma}

\begin{proof}
Parseval gives
\[
\sum_{j=0}^{M-1}\abs{\widehat{\mathbf 1_A}^{(M)}(j)}^2=M|A|.
\]
Applying Cauchy--Schwarz,
\[
\sum_{j=0}^{M-1}\abs{\widehat{\mathbf 1_A}^{(M)}(j)}
\le \sqrt{M}\sqrt{M|A|}=M\sqrt{|A|}.
\]
Dividing by $M$ gives the result.
\end{proof}

\begin{lemma} \label{lem:completion}
Let $S\in E(\F_q)$ has order $M$, and let $H_S=\la S\ra$.
For any subset $A\subset \Z/M\Z$ and any function $F:H_S\to\mathbb C$, we have:
\begin{equation}\label{eq:completion}
\sum_{u\in A}F([u]S)
=
\frac{1}{M}\sum_{j=0}^{M-1}\widehat{\mathbf 1_A}^{(M)}(j)
\sum_{P\in H_S}\omega_j(P)F(P),
\end{equation}
where for $j=0,1\cdots,M-1$, $\omega_j$ is the group character of $H_S$ given by $\omega_j([u]S):=e^{2\pi iju/M}$.
\end{lemma}

\begin{proof}
This is the orthogonality relation on $\Z/M\Z$:
\[
\mathbf 1_A(u)=\frac{1}{M}\sum_{j=0}^{M-1}\widehat{\mathbf 1_A}^{(M)}(j)e^{2\pi iju/M}.
\]
Substituting this expression into: 
\[
\sum_{u\in A}F([u]S) = \sum_{u \in \Z/M \Z} \mathbf 1_A(u)F([u]S)
\]
yields \eqref{eq:completion}.
\end{proof}

\begin{proposition}[Orbit-sum bound via Fourier $\ell^1$ mass]\label{prop:sum-L1}
Let $S\in E(\F_q)$ has order $M$, let $A\subset \Z/M\Z$ be an arbitrary subset, and let $f\in\F_q(E)$ be a rational function that is non-degenerate.
Then
\begin{equation}\label{eq:sum-L1}
\abs{\sum_{\substack{u\in A\\ f([u]S)\neq \infty}}\psi_1(f([u]S))}
\le 2\deg(f)\,q^{1/2}\,\mathcal L_M(A).
\end{equation}
If the polar divisor of $f$ is supported on a single prime divisor of $E$, one may replace $2\deg(f)$ by $1+\deg(f)$.
\end{proposition}

\begin{proof}
Define
\[
F(P):=
\begin{cases}
\psi_1(f(P)), & f(P)\neq \infty,\\
0, & f(P)=\infty.
\end{cases}
\]
Applying Lemma~\ref{lem:completion} with this $F$, the inner sum becomes exactly the finite-value twisted subgroup sum appearing in Lemma~\ref{lem:KS} (with the subgroup of $E(\F_q)$ being taken to be $H_S$ and the group character being $\omega_j$, $j=0,1,\cdots,M-1$).
Taking absolute values and summing with weights $\abs{\widehat{\mathbf 1_A}^{(M)}(j)}/M$ yields \eqref{eq:sum-L1}.
\end{proof}

\subsection{Admissible index sets}

Recall the identity ($n \geq 0$):
\[
P_n = P_0 + [\beta_n] R
\]
and more generally ($n,\iota \geq 0$): 
\[
P_{n + \iota} = P_{\iota} + [e^{\iota} \beta_n] R
\]

\bigskip
\noindent For one-point discrepancy, define
\[
B_{\mathrm{adm}}:=\{u\in B:\ P_0+[u]R\neq O\}.
\]
Then $\mathcal{P}$ is the image of $B_{\mathrm{adm}}$ under the injective map:
\[
u \in B \longmapsto
G(P_0+[u]R)  \in [0,1]^{2r}.
\]
and so
\[
|B_{\mathrm{adm}}|= |\mathcal{P}| =N^{\ast}, \qquad |B \setminus B_{\mathrm{adm}} | \leq 1
\]

\noindent For serial discrepancy of length $s$, define
\[
B_{s,\mathrm{adm}}:=\{u\in B:\ P_{\iota}+[e^{\iota} u]R\neq O\ \text{for all }0\le \iota\le s-1\}.
\]
Then $\mathcal{P}^{(s)}$ is the image of $B_{s,\mathrm{adm}}$ under the injective map:
\[
u \in B \longmapsto
\big(G(P_0+[u]R),G(P_1 +[eu]R),\dots,G(P_{s-1}+[e^{s-1}u]R)\big) \in [0,1]^{2rs}.
\]
and so
\[
|B_{s,\mathrm{adm}}|= |\mathcal{P}^{(s)}|  =N_s^{\ast}, \qquad |B \setminus  B_{s,\mathrm{adm}}| \leq s
\]

\bigskip

\noindent For the non-overlapping variant, fix $s\ge 2$ and keep the notation
\[
R_s:=P_s-P_0, \qquad m_s:=\ord(R_s)
\]
\[
\gamma_n:=\frac{e^{ns}-1}{e^s-1}\in\Z\qquad (n\ge 0; \gamma_0=0)
\]
\[
L_s:=t/\gcd(s,t).
\]
and recall that $L_s$ is the period of the sequence $\{P_{ns}\}_{n \geq 0}$.

\bigskip
Also recall that:
\[
P_{ns}=P_0+[\gamma_n]R_s,
\qquad
\gamma_{n+1}\equiv e^s\gamma_n+1\pmod{m_s},
\qquad
\gamma_0\equiv 0 \pmod{m_s}.
\]
from which one deduces (as seen before) that $L_s$ is equal to the period of the LCG $\{\gamma_n \bmod{m_s}\}_{n \geq 0}$. 

\bigskip

Set
\[
\widetilde{B}_s:=\{\gamma_0,\dots,\gamma_{L_s-1} \bmod{m_s}\}\subset \Z/m_s\Z.
\]
and define:

\[
\widetilde{B}_{s,\mathrm{adm}}:=\{u\in \widetilde{B}_s:\ P_{\iota}+[e^{\iota} u]R_s\neq O\ \text{for all }0\le \iota\le s-1\}.
\]

\bigskip

\noindent Again recall the more general identity ($n,\iota \geq 0$):
\[
P_{ns + \iota} = P_{\iota} + [e^{\iota} \gamma_n]R_s
\]
\noindent and thus $\widetilde{\mathcal{P}}^{(s)}$ is the image of $ \widetilde{B}_{s,\mathrm{adm}}$ under the injective map:

\[
u \in \widetilde{B}_s \longmapsto
\big(G(P_0+[u]R_s),  G(P_1+ [e u]R_s) , \dots,G(P_{s-1}+[e^{s-1}u]R_s)\big) \in [0,1]^{2rs}.
\]
and so
\[
|\widetilde B_{s,\mathrm{adm}}|= |\widetilde{\mathcal{P}}^{(s)}|  =\widetilde{N}_s^{\ast}, \qquad |\widetilde{B}_s \setminus  \widetilde B_{s,\mathrm{adm}}| \leq s/\gcd(s,t)
\]

\begin{lemma}\label{lem:L1-remove}
Let $M\ge 1$ and let $A\subseteq B\subseteq \Z/M\Z$.
Then
\[
\mathcal L_M(A)\le \mathcal L_M(B)+|B\setminus A|^{1/2}.
\]
In particular,
\[
\mathcal L_m(B_{\mathrm{adm}})\le \mathcal L_m(B)+1,
\]
\[
\mathcal L_m(B_{s,\mathrm{adm}})\le \mathcal L_m(B)+s^{1/2},
\]
\[
\mathcal L_{m_s}(\widetilde{B}_{s,\mathrm{adm}})\le \mathcal L_{m_s}(\widetilde{B}_s)+\big(s/\gcd(s,t)\big)^{1/2},
\]
\end{lemma}

\begin{proof}
Write $B=A\sqcup C$ with $C=B\setminus A$.
Then
\[
\widehat{\mathbf 1_A}^{(M)}(j)=\widehat{\mathbf 1_B}^{(M)}(j)-\widehat{\mathbf 1_C}^{(M)}(j),
\]
so
\[
\abs{\widehat{\mathbf 1_A}^{(M)}(j)}
\le
\abs{\widehat{\mathbf 1_B}^{(M)}(j)}+\abs{\widehat{\mathbf 1_C}^{(M)}(j)}.
\]
Summing over $j$ and dividing by $M$ gives
\[
\mathcal L_M(A)\le \mathcal L_M(B)+\mathcal L_M(C).
\]
Apply Lemma~\ref{lem:L1-sqrt} to obtain $\mathcal L_M(C) \leq |C|^{1/2}$ and we are done.
\end{proof}

\begin{remark}
If the orbit $\{P_n\}_{n \ge 0}$ avoids $O$, then $B=B_{\mathrm{adm}}=B_{s,\mathrm{adm}}  $, and $\widetilde{B}_s=\widetilde{B}_{s,\mathrm{adm}}$, and so the boundary correction terms:
\[
+1, \quad \quad +s^{1/2}, \quad  \quad + (s/\gcd(s,t))^{1/2}
\]
in Lemma~\ref{lem:L1-remove} are not needed.
\end{remark}

\subsection{Discrepancy bound in the general sub-period regime}

\begin{theorem}[Discrepancy bound via admissible Fourier mass]\label{thm:CaseC-D-L1}
Let
\[
\mathcal P:=\{ G(P_n)\}_{\substack{0\le n < t\\ P_n\neq O}},
\qquad N^{\ast}:=|\mathcal P|.
\]
Assume $p\ge 5$.
Then
\begin{equation}\label{eq:CaseC-D-L1}
D(\mathcal P)
\le 1-\Big(1-\frac{1}{p^a}\Big)^{2r}
+\frac{4q^{1/2}}{N^{\ast}}\,\mathcal L_m(B_{\mathrm{adm}})\,\big(C_0\ln(p^a)+1\big)^{2r}.
\end{equation}
Consequently,
\begin{equation}\label{eq:CaseC-D-L1-rough}
D(\mathcal P)
\le 1-\Big(1-\frac{1}{p^a}\Big)^{2r}
+\frac{4q^{1/2}}{N^{\ast}}\,\big(\mathcal L_m(B)+1\big)\,\big(C_0\ln(p^a)+1\big)^{2r}.
\end{equation}
\end{theorem}

\begin{proof}
We know that $\mathcal{P}$ is the image of $B_{\mathrm{adm}}$ under 
\[
u \in B \longmapsto G(P_0 +[u] R) \in [0,1]^{2r}
\]
So given Walsh frequency ${\bf k} \in\Delta_{2r}^{\ast}$,
by Lemma~\ref{lem:Walsh-to-char}, we have
\[
N^{\ast}S(w_{\bf k},\mathcal P)
=
\sum_{u\in B_{\mathrm{adm}}}
\psi_1\big(\eta_{\bf k} x(P_0+[u]R)+\widetilde\eta_{\bf k} y(P_0+[u]R)\big).
\]
Hence
\[
N^{\ast}S(w_{\bf k},\mathcal P)
=\sum_{u\in B_{\mathrm{adm}}} \psi_1\big(f_{{\bf k},P_0}([u] R)   \big).
\]
where $f_{{\bf k},P_0}$ is as in Theorem~\ref{thm:CaseB-D}, where we have shown that it is non-degenerate and $\deg(f_{{\bf k},P_0}) \leq 3 $.
Applying Proposition~\ref{prop:sum-L1} with $f=f_{{\bf k},P_0}$, $S=R$, $M=m$, and $A=B_{\mathrm{adm}}$ gives
\[
\abs{N^{\ast}S(w_{\bf k},\mathcal P)}
\le (1+\deg(f_{{\bf k},P_0}))q^{1/2}\,\mathcal L_m(B_{\mathrm{adm}})
\le 4q^{1/2}\mathcal L_m(B_{\mathrm{adm}}).
\]
Applying Lemma~\ref{lem:Walsh-discrepancy} with $d=2r$ yields \eqref{eq:CaseC-D-L1}.
The rougher estimate \eqref{eq:CaseC-D-L1-rough} follows from Lemma~\ref{lem:L1-remove}.
\end{proof}

\subsection{Serial discrepancy in the general sub-period regime}

\begin{theorem}[Serial discrepancy bound via admissible Fourier mass]\label{thm:CaseC-Ds-L1}
Assume $|e|\ge 2$, $\gcd(e,p)=1$, and $p\ge 5$.
Fix $2\le s<t$ and let $\mathcal P^{(s)}$ be the admissible serial $s$-tuple collection, of cardinality $N_s^{\ast}$.
Then
\begin{equation}\label{eq:CaseC-Ds-L1}
D(\mathcal P^{(s)})
\le 1-\Big(1-\frac{1}{p^a}\Big)^{2rs}
+\frac{6q^{1/2}}{N_s^{\ast}}\Big(\sum_{\iota=0}^{s-1}e^{2 \iota}\Big)\mathcal L_m(B_{s,\mathrm{adm}})\,\big(C_0\ln(p^a)+1\big)^{2rs}.
\end{equation}
Consequently,
\begin{equation}\label{eq:CaseC-Ds-L1-rough}
D(\mathcal P^{(s)})
\le 1-\Big(1-\frac{1}{p^a}\Big)^{2rs}
+\frac{6q^{1/2}}{N_s^{\ast}}\Big(\sum_{\iota=0}^{s-1}e^{2 \iota}\Big)\big(\mathcal L_m(B)+s^{1/2}\big)\,\big(C_0\ln(p^a)+1\big)^{2rs}.
\end{equation}
\end{theorem}

\begin{proof}
We know that $\mathcal{P}^{(s)}$ is the image of $B_{s,\mathrm{adm}}$ under the injective map:
\[
u \in B \longmapsto \big(G(P_0+[u]R),G(P_1+[eu]R),\dots,G(P_{s-1}+[e^{s-1}u]R)\big) \in [0,1]^{2rs}
\]
So given Walsh frequency ${\bf k}\in\Delta_{2rs}^{\ast}$, we have as in the proof of Theorem~\ref{thm:CaseB-Ds}:
\[
N_s^{\ast}S(w_{\bf k},\mathcal P^{(s)}) = \sum_{u \in B_{s,\mathrm{adm}}} \psi_1\big(g_{\bf k}([u]R)\big)
\]
where $g_{\bf k}$ is the same function in the proof of Theorem~\ref{thm:CaseB-Ds}, where we have shown that it is non-degenerate and
\[
\deg(g_{\bf k})\le 3\sum_{\iota=0}^{s-1}e^{2 \iota}
\]
Applying Proposition~\ref{prop:sum-L1} with $f=g_{\bf k}$, $S=R$, $M=m$, and $A=B_{s,\mathrm{adm}}$ yields
\[
\abs{N_s^{\ast}S(w_{\bf k},\mathcal P^{(s)})}
\le 2\deg(g_{\bf k})q^{1/2}\,\mathcal L_m(B_{s,\mathrm{adm}})
\le 6q^{1/2}\Big(\sum_{\iota=0}^{s-1}e^{2 \iota}\Big)\mathcal L_m(B_{s,\mathrm{adm}}).
\]
Applying Lemma~\ref{lem:Walsh-discrepancy} in dimension $2rs$ yields \eqref{eq:CaseC-Ds-L1}.
The rough bound \eqref{eq:CaseC-Ds-L1-rough} follows from Lemma~\ref{lem:L1-remove}.
\end{proof}

\subsection{Non-overlapping discrepancy in the general sub-period regime}

\begin{theorem}[Non-overlapping discrepancy bound via admissible Fourier mass]\label{thm:CaseC-Dtilde-L1}
Assume $|e|\ge 2$, $\gcd(e,p)=1$, and $p\ge 5$.
Fix $s\ge 2$ and let $\widetilde{\mathcal P}^{(s)}$ be the admissible non-overlapping $s$-tuple collection, of cardinality $\widetilde N_s^{\ast}$.
Then
\begin{equation}\label{eq:CaseC-Dtilde-L1}
D(\widetilde{\mathcal P}^{(s)})
\le 1-\Big(1-\frac{1}{p^a}\Big)^{2rs}
+\frac{6q^{1/2}}{\widetilde N_s^{\ast}}\Big(\sum_{\iota=0}^{s-1}e^{2 \iota}\Big)\mathcal L_{m_s}(\widetilde B_{s,\mathrm{adm}})\,\big(C_0\ln(p^a)+1\big)^{2rs}.
\end{equation}
Consequently,
\begin{equation}\label{eq:CaseC-Dtilde-L-rough}
D(\widetilde{\mathcal P}^{(s)})
\le 1-\Big(1-\frac{1}{p^a}\Big)^{2rs}
+\frac{6q^{1/2}}{\widetilde N_s^{\ast}}\Big(\sum_{\iota=0}^{s-1}e^{2 \iota}\Big) \Big(\mathcal L_{m_s}(\widetilde B_{s}) + \big( s/\gcd(s,t)\big)^{1/2} \Big)\,\big(C_0\ln(p^a)+1\big)^{2rs}
\end{equation}
\end{theorem}

\begin{proof}
We know that $\widetilde{\mathcal{P}}^{(s)}$ is the image of $\widetilde{B}_{s,\mathrm{adm}}$ under the injective map:
\[
u \in \widetilde{B}_s \longmapsto \big(G(P_0+[u]R_s),G(P_1+[eu]R_s),\dots,G(P_{s-1}+[e^{s-1}u]R_s)\big) \in [0,1]^{2rs}
\]
So given Walsh frequency ${\bf k}\in\Delta_{2rs}^{\ast}$, we have as in the proof of Theorem~\ref{thm:CaseB-Dtilde}:
\[
\widetilde{N}_s^{\ast}S(w_{\bf k},\widetilde{\mathcal P}^{(s)}) = \sum_{u \in \widetilde{B}_{s,\mathrm{adm}}   } \psi_1\big(g_{\bf k}([u]R_s)\big)
\]

Applying Proposition~\ref{prop:sum-L1} with $f=g_{\bf k}$, $S=R_s$, $M=m_s$, and $A=\widetilde{B}_{s,\mathrm{adm}} $ yields
\[
\abs{\widetilde{N}_s^{\ast}S(w_{\bf k},\widetilde{\mathcal P}^{(s)})}
\le 2\deg(g_{\bf k})q^{1/2}\,\mathcal L_{m_s}(\widetilde{B}_{s,\mathrm{adm}}  )
\le 6q^{1/2}\Big(\sum_{\iota=0}^{s-1}e^{2 \iota}\Big)\mathcal L_{m_s}(\widetilde{B}_{s,\mathrm{adm}} ).
\]

Applying Lemma~\ref{lem:Walsh-discrepancy} in dimension $2rs$ yields \eqref{eq:CaseC-Dtilde-L1}. The rough bound \eqref{eq:CaseC-Dtilde-L-rough} again follows from Lemma~\ref{lem:L1-remove}.
\end{proof}

\section{Arithmetic structure of the index sequence}

Theorem~\ref{thm:CaseC-D-L1}, Theorem~\ref{thm:CaseC-Ds-L1}, and Theorem~\ref{thm:CaseC-Dtilde-L1} reduce the discrepancy problem to admissible Fourier masses of one-dimensional LCG index sets.
This section isolates the arithmetic structure of those index sets.

\subsection{Local behavior modulo prime powers}

Write the prime-power factorization
\[
m=\prod_{\ell}\ell^{\nu_\ell}.
\]
For a prime $\ell\mid m$, let
\[
a_\ell:=v_\ell(e-1).
\]
where $v_{\ell}$ is the normalized valuation at the prime $\ell$.

\begin{lemma}[Local behavior modulo prime powers]\label{lem:local-behavior}
Let $\ell^{\nu}\parallel m$ (so $\nu = \nu_{\ell}$) and set $a:=a_{\ell}$.
Then:
\begin{enumerate}[label=(\alph*),leftmargin=*]
\item For every $b\le \min\{a,\nu\}$, one has
\[
\beta_n\equiv n\pmod{\ell^b}.
\]
In particular, if $a\ge \nu$, then $\beta_n\equiv n\pmod{\ell^\nu}$ for all $n$.
\item If $a=0$, then there exists a unique fixed point $\beta_{\ast,\ell} \bmod{\ell^{\nu}}\in (\Z/\ell^\nu\Z)^{\times}$ such that
\[
\beta_{\ast,\ell}\equiv e\beta_{\ast,\ell}+1\pmod{\ell^\nu},
\]
and
\[
\beta_n-\beta_{\ast,\ell}\equiv -\beta_{\ast,\ell}e^n\pmod{\ell^\nu}.
\]
\end{enumerate}
\end{lemma}

\begin{proof}
(a) As $e\equiv 1\pmod{\ell^b}$, so the recursion \eqref{eq:index-LCG} becomes
\[
\beta_{n+1}\equiv \beta_n+1\pmod{\ell^b},
\qquad \beta_0\equiv 0,
\]
hence $\beta_n\equiv n\pmod{\ell^b}$.

(b) If $a=0$, then $1-e$ is invertible modulo $\ell^\nu$, so the fixed-point congruence has a unique solution $\beta_{\ast,\ell} \bmod{\ell^{\nu}} \in (\Z/ \ell^{\nu} \Z)^{\times}$.
Subtracting the fixed-point relation from the recursion relation gives
\[
(\beta_{n+1}-\beta_{\ast,\ell})\equiv e(\beta_n-\beta_{\ast,\ell})\pmod{\ell^\nu},
\]
and iteration yields the stated formula.
\end{proof}

\begin{corollary}[No mixed local factors regime]\label{cor:no-mixed}
Assume that for every prime $\ell\mid m$, one has either $v_\ell(e-1)=0$ or $v_\ell(e-1)\ge \nu_\ell$.
Define
\[
m_{\mathrm{tr}}:=\prod_{v_\ell(e-1)\ge \nu_\ell}\ell^{\nu_\ell},
\qquad
m_{\mathrm{pow}}:=\prod_{v_\ell(e-1)=0}\ell^{\nu_\ell},
\]
Then $m=m_{\mathrm{tr}}m_{\mathrm{pow}}$, $\gcd(m_{\mathrm{tr}},m_{\mathrm{pow}})=1$, and there exists $\beta_{\ast} \bmod{ m_{\mathrm{pow}}  }\in (\Z/m_{\mathrm{pow}}\Z)^{\times}$ such that under the Chinese Remainder Theorem isomorphism
\[
\Z/m\Z\cong \Z/m_{\mathrm{tr}}\Z\times \Z/m_{\mathrm{pow}}\Z,
\]
one has
\[
\beta_n\longmapsto \big(n\bmod m_{\mathrm{tr}},\ \beta_{\ast}(1-e^n)\bmod m_{\mathrm{pow}}\big).
\]
\end{corollary}

\begin{proof}
For every prime power $\ell^{\nu_{\ell}}$ dividing $m_{\mathrm{tr}}$, part (a) of Lemma~\ref{lem:local-behavior} gives $\beta_n\equiv n   \bmod{\ell^{\nu_{\ell}}}$.
For every prime power $\ell^{\nu_{\ell}}$ dividing $m_{\mathrm{pow}}$, part (b) gives a local fixed point $\beta_{\ast,\ell} \bmod{\ell^{\nu_{\ell}}} \in (\Z/\ell^{\nu_{\ell}} \Z)^{\times} $ and hence
\[
\beta_n\equiv \beta_{\ast,\ell}(1-e^n)\pmod{\ell^{\nu_\ell}}.
\]
The Chinese Remainder Theorem then yields unique $$\beta_{\ast}  \bmod{m_{\mathrm{pow}}} \in (\Z/m_{\mathrm{pow}}\Z)^{\times}$$ such that  $\beta_{\ast}  \equiv  \beta_{\ast,\ell} \bmod{\ell^{\nu_{\ell}}}$ for every prime $\ell$ dividing $m_{\mathrm{pow}}$, in other words $\beta_{\ast}  \bmod{m_{\mathrm{pow}}}$ is the unique solution of the fixed point congruence: 
\[
\beta_{\ast} =e \beta_{\ast} +1 \bmod{m_{\mathrm{pow}}}
\]
and the claimed formula follows.
\end{proof}

\begin{remark}\label{rem:generic-local}
The proofs of Lemma~\ref{lem:local-behavior} and Corollary~\ref{cor:no-mixed} depend only on the fact that the recursion is of the form:
\[
x_{n+1}\equiv a x_n+1\pmod M.
\]
Accordingly, the same arguments apply verbatim to any such recursion after replacing $(e,m,\{  \beta_n  \})$ by $(a,M,\{x_n\})$.
This observation applies in particular to the setting $(e^s,m_s,\{\gamma_n\})$, and we will use similar notations below with $(e^s,m_s,\{\gamma_n\})$ in place of $(e,m,\{\beta_n\})$.
\end{remark}

\begin{remark}
Lemma~\ref{lem:local-behavior} shows that the unresolved arithmetic is concentrated in the \emph{mixed} local prime powers, namely those with
\[
0<v_\ell(e-1)<\nu_\ell.
\]
If no such factors occur, the index sequence $\{ \beta_n \bmod{m} \}$ is explicitly described by Corollary~\ref{cor:no-mixed} as a Chinese Remainder Theorem combination of a translation component and a power-generator component.
\end{remark}

\begin{remark}[Conventions for trivial power components]\label{rem:trivial-power-component}
If $m_{\mathrm{pow}}=1$, then the pure power component is trivial.
In that case we set
\[
\tau:=1,\qquad \mathcal{G}:=\{0\}\subset \Z/1\Z.
\]
Similarly, if $m_{s,\mathrm{pow}}=1$, we set
\[
\tau_s:=1,\qquad \mathcal{G}_s:=\{0\}\subset \Z/1\Z.
\]
With these conventions, all subsequent formulas remain valid without further case distinctions.
\end{remark}

\begin{theorem}\label{thm:no-mixed-improvement}
Assume the hypotheses of Corollary~\ref{cor:no-mixed}.
If $m_{\mathrm{pow}}=1$, set
\[
\tau:=1,\qquad \mathcal{G}:=\{0\}\subset \Z/1\Z.
\]
If $m_{\mathrm{pow}}>1$, set $\tau$ to be equal to the order of $e \bmod{m_{\mathrm{pow}}} $ in the multiplicative group $(\Z/ m_{\mathrm{pow}} \Z)^{\times}$ (recall by Lemma~\ref{lem:gcd-em} we have that $e$ is relatively prime to $m$, hence relatively prime to $m_{\mathrm{pow}}$), and $\mathcal{G} \subset \Z/m_{\mathrm{pow}}\Z$ the subset:
\[
\mathcal{G}:=\{\beta_{\ast}(1-e^v)\bmod m_{\mathrm{pow}}:0\le v<\tau\}\subset \Z/m_{\mathrm{pow}}\Z.
\]
(so $|\mathcal{G}|=\tau$). Assume in addition that
\[
\gcd(m_{\mathrm{tr}},\tau)=1.
\]
Then the following hold:

\begin{enumerate}[label=(\roman*),leftmargin=*]
\item We have $t= m_{\mathrm{tr}}\tau$.

\item Under the Chinese Remainder Theorem isomorphism:
\[
\Z/m\Z\cong \Z/m_{\mathrm{tr}}\Z\times \Z/m_{\mathrm{pow}}\Z,
\]
the cycle $B$ of the LCG $\{\beta_n \bmod{m}\}_{n \geq 0}$:
\[
B=\{\beta_0,\dots,\beta_{t-1} \bmod{m}\}
\]
is exactly the Cartesian product
\[
B=\Z/m_{\mathrm{tr}}\Z\times \mathcal{G}.
\]
\item The normalized Fourier $\ell^1$ mass satisfies
\[
\mathcal L_m(B)=\mathcal L_{m_{\mathrm{pow}}}(\mathcal{G}).
\]
Consequently,
\[
\mathcal L_m(B_{\mathrm{adm}})\le \mathcal L_{m_{\mathrm{pow}}}(\mathcal{G})+1,
\qquad
\mathcal L_m(B_{s,\mathrm{adm}})\le \mathcal L_{m_{\mathrm{pow}}}(\mathcal{G})+s^{1/2}.
\]
\end{enumerate}
\end{theorem}

\begin{proof}
If $m_{\mathrm{pow}}=1$, then $m_{\mathrm{tr}}=m$, $\tau=1$, and $\mathcal{G}=\{0\}$.
Corollary~\ref{cor:no-mixed} gives $\beta_n\equiv n\pmod m$, so the period is $t=m=m_{\mathrm{tr}}\tau$, the cycle is $B=\Z/m\Z=\Z/m_{\mathrm{tr}}\Z\times \mathcal{G}$, and $\mathcal L_m(B)=1=\mathcal L_1(\mathcal{G})$.
The remaining claims are immediate from Lemma~\ref{lem:L1-remove}.
Thus assume $m_{\mathrm{pow}}>1$. By Corollary~\ref{cor:no-mixed}, under the Chinese Remainder Theorem isomorphism one has:
\[
\beta_n \bmod{m} \longmapsto \big(n\bmod m_{\mathrm{tr}},\ \beta_{\ast}(1-e^n)\bmod m_{\mathrm{pow}}\big).
\]

The first coordinate has period $m_{\mathrm{tr}}$.
The second coordinate has period $\tau$: indeed, if
\[
\beta_{\ast}(1-e^{n+r})\equiv \beta_{\ast}(1-e^n)\pmod{m_{\mathrm{pow}}},
\]
then
\[
\beta_{\ast}e^n(e^r-1)\equiv 0\pmod{m_{\mathrm{pow}}}.
\]
Therefore as both $e$ and $\beta_{\ast}$ are invertible modulo $m_{\mathrm{pow}}$ we have:
\[
e^r\equiv 1\pmod{m_{\mathrm{pow}}},
\]
so $\tau\mid r$.
Hence the second coordinate has exact period $\tau$. Thus as $t$ is the period of  $\{\beta_n \bmod{m} \}_{n \geq 0} $ it follows that:

\[
t=\operatorname{lcm}(m_{\mathrm{tr}},\tau)=m_{\mathrm{tr}}\tau,
\]
because $\gcd(m_{\mathrm{tr}},\tau)=1$.
This proves (i).

\bigskip
To prove (ii), let $(u,v)\in \Z/m_{\mathrm{tr}}\Z\times \{0,\dots,\tau-1\}$.
By the Chinese Remainder Theorem (using the assumption that $m_{\mathrm{tr}}$ and $\tau$ are relatively prime) there exists a unique residue class $n\bmod t$ such that
\[
n\equiv u\pmod{m_{\mathrm{tr}}},
\qquad
n\equiv v\pmod{\tau}.
\]
Since $e^n\equiv e^v\pmod{m_{\mathrm{pow}}}$, Corollary~\ref{cor:no-mixed} gives
\[
\beta_n \bmod{m} \longmapsto \big(u \bmod{m_{\mathrm{tr}}},\ \beta_{\ast}(1-e^v)  \bmod{m_{\mathrm{pow}}}\big).
\]
Thus every element of $\Z/m_{\mathrm{tr}}\Z\times \mathcal{G}$ occurs in $B$.
Conversely, we already know that every $\beta_n \bmod{m}$ has this form under the Chinese Remainder Theorem isomorphism, thus proving
\[
B=\Z/m_{\mathrm{tr}}\Z\times \mathcal{G}.
\]

For (iii), observe first that the normalized Fourier $\ell^1$ mass is invariant under additive-group isomorphisms.
Hence one may compute $\mathcal L_m(B)$ on the product group
\[
\Z/m_{\mathrm{tr}}\Z\times \Z/m_{\mathrm{pow}}\Z
\]
using product characters
\[
\chi_{a,b}(x,y):=\exp\!\Big(-2\pi i\frac{ax}{m_{\mathrm{tr}}}\Big)\exp\!\Big(-2\pi i\frac{by}{m_{\mathrm{pow}}}\Big),
\]
with $0\le a<m_{\mathrm{tr}}$ and $0\le b<m_{\mathrm{pow}}$.
Since $B=\Z/m_{\mathrm{tr}}\Z\times \mathcal{G}$, one has
\[
\widehat{\mathbf 1_B}^{(m)}(a,b)
=
\sum_{x\in \Z/m_{\mathrm{tr}}\Z} e^{-2\pi i ax/m_{\mathrm{tr}}}
\sum_{y\in \mathcal{G}} e^{-2\pi i by/m_{\mathrm{pow}}}.
\]
The first factor vanishes unless $a=0$, in which case it equals $m_{\mathrm{tr}}$.
Therefore
\[
\sum_{a=0}^{m_{\mathrm{tr}}-1}\sum_{b=0}^{m_{\mathrm{pow}}-1}
\big|\widehat{\mathbf 1_B}^{(m)}(a,b)\big|
=
m_{\mathrm{tr}}
\sum_{b=0}^{m_{\mathrm{pow}}-1}\big|\widehat{\mathbf 1_{\mathcal{G}}}^{(m_{\mathrm{pow}})}(b)\big|.
\]
Dividing by $m=m_{\mathrm{tr}}m_{\mathrm{pow}}$ yields
\[
\mathcal L_m(B)=\mathcal L_{m_{\mathrm{pow}}}(\mathcal{G}).
\]
The final inequalities now follow from Lemma~\ref{lem:L1-remove}.
\end{proof}

\begin{corollary}[Improved discrepancy bounds in the no-mixed regime]\label{cor:no-mixed-discrepancy}
Under the hypotheses of Theorem~\ref{thm:no-mixed-improvement},
\begin{align}
D(\mathcal P)
&\le
1-\Big(1-\frac{1}{p^a}\Big)^{2r}
+\frac{4q^{1/2}}{N^{\ast}}\big(\mathcal L_{m_{\mathrm{pow}}}(\mathcal{G})+1\big)\big(C_0\ln(p^a)+1\big)^{2r},\label{eq:no-mixed-D}\\[0.5em]
D(\mathcal P^{(s)})
&\le
1-\Big(1-\frac{1}{p^a}\Big)^{2rs}
+\frac{6q^{1/2}}{N_s^{\ast}}\Big(\sum_{\iota=0}^{s-1}e^{2\iota}\Big)\big(\mathcal L_{m_{\mathrm{pow}}}(\mathcal{G})+s^{1/2}\big)\big(C_0\ln(p^a)+1\big)^{2rs}.\label{eq:no-mixed-Ds}
\end{align}
In particular, since $\mathcal L_{m_{\mathrm{pow}}}(\mathcal{G})\le |\mathcal{G}|^{1/2} = \tau^{1/2}$ by Lemma~\ref{lem:L1-sqrt}, we have:
\begin{align}
D(\mathcal P)
&\le
1-\Big(1-\frac{1}{p^a}\Big)^{2r}
+\frac{4q^{1/2}}{N^{\ast}}\big(\tau^{1/2}+1\big)\big(C_0\ln(p^a)+1\big)^{2r},\label{eq:no-mixed-D-rough}\\[0.5em]
D(\mathcal P^{(s)})
&\le
1-\Big(1-\frac{1}{p^a}\Big)^{2rs}
+\frac{6q^{1/2}}{N_s^{\ast}}\Big(\sum_{\iota=0}^{s-1}e^{2 \iota}\Big)\big(\tau^{1/2} +s^{1/2} \big)\big(C_0\ln(p^a)+1\big)^{2rs}.\label{eq:no-mixed-Ds-rough}
\end{align}
\end{corollary}

\begin{proof}
Combine Theorem~\ref{thm:no-mixed-improvement} with Theorem~\ref{thm:CaseC-D-L1}, Theorem~\ref{thm:CaseC-Ds-L1}.
\end{proof}

\begin{theorem}\label{thm:no-mixed-improvement-derived}
Fix $s\ge 2$ and as before: 
\[
R_s:=P_s-P_0 =[\beta_s]R,
\]
\[
m_s:=\ord(R_s),\qquad \widetilde{B}_s:=\{\gamma_0,\dots,\gamma_{L_s-1} \bmod{m_s}\}\subset \Z/m_s\Z,
\]
where
\[
\gamma_{n+1}\equiv e^s\gamma_n+1\pmod{m_s},\qquad \gamma_0\equiv 0,
\qquad L_s:=t/\gcd(t,s).
\]
Assume that for every prime $\ell\mid m_s$, one has either $v_\ell(e^s-1)=0$ or $v_\ell(e^s-1)\ge \nu_{\ell,s}$, where
\[
m_s=\prod_\ell \ell^{\nu_{\ell,s}}.
\]
Define
\[
m_{s,\mathrm{tr}}:=\prod_{v_\ell(e^s-1)\ge \nu_{\ell,s}}\ell^{\nu_{\ell,s}},
\qquad
m_{s,\mathrm{pow}}:=\prod_{v_\ell(e^s-1)=0}\ell^{\nu_{\ell,s}},
\]
so $m_s=m_{s,\mathrm{tr}}m_{s,\mathrm{pow}}$ and $\gcd(m_{s,\mathrm{tr}},m_{s,\mathrm{pow}})=1$.

\bigskip
If $m_{s,\mathrm{pow}}=1$, set
\[
\tau_s:=1,\qquad \mathcal{G}_s:=\{0\}\subset \Z/1\Z.
\]

If $m_{s,\mathrm{pow}}>1$, let $\gamma_{\ast,s} \bmod{m_{s,\mathrm{pow}}} \in (\Z/m_{s,\mathrm{pow}}\Z)^{\times}$ be the unique solution of
\[
\gamma_{\ast,s}\equiv e^s\gamma_{\ast,s}+1\pmod{m_{s,\mathrm{pow}}},
\]

Set $\tau_s$ to be equal to the order of $e^s \bmod{ m_{s,\mathrm{pow}} }$ in the multiplicative group   $ (\Z/ m_{s,\mathrm{pow}} \Z)^{\times}$, and
\[
\mathcal{G}_s:=\{\gamma_{\ast,s}(1-(e^s)^v)\bmod m_{s,\mathrm{pow}}:0\le v<\tau_s\}.
\]
(and so $|\mathcal{G}_s|=\tau_s$). Assume moreover that
\[
\gcd(m_{s,\mathrm{tr}},\tau_s)=1.
\]
Then the following hold:
\begin{enumerate}[label=(\roman*),leftmargin=*]
\item We have $L_s =m_{s,\mathrm{tr}}\tau_s$.

\item Under the Chinese Remainder Theorem isomorphism:
\[
\Z/m_s\Z\cong \Z/m_{s,\mathrm{tr}}\Z\times \Z/m_{s,\mathrm{pow}}\Z,
\]
we have
\[
\widetilde{B}_s=\Z/m_{s,\mathrm{tr}}\Z\times \mathcal{G}_s.
\]
\item The normalized Fourier $\ell^1$ mass satisfies
\[
\mathcal L_{m_s}(\widetilde{B}_s)=\mathcal L_{m_{s,\mathrm{pow}}}(\mathcal{G}_s).
\]
Consequently,
\[
\mathcal L_{m_s}(\widetilde B_{s,\mathrm{adm}})
\le \mathcal L_{m_{s,\mathrm{pow}}}(\mathcal{G}_s)+\big( s/\gcd(s,t) \big)^{1/2}.
\]
\end{enumerate}
\end{theorem}

\begin{proof}
If $m_{s,\mathrm{pow}}=1$, then $\tau_s=1$ and $\mathcal{G}_s=\{0\}$.
In this case the no-mixed hypothesis forces $m_s=m_{s,\mathrm{tr}}$, and the same analysis as in the proof of Lemma~\ref{lem:local-behavior} and Corollary~\ref{cor:no-mixed}, applied to the recursion:
\[
\gamma_{n+1}\equiv e^s\gamma_n+1\pmod{m_s}
\]
shows that
\[
\gamma_n\equiv n\pmod{m_s}.
\]
Hence as $L_s$ is the period of $\{\gamma_n \bmod{m_s}\}_{n \geq 0}$, we have
\[
L_s=m_s=m_{s,\mathrm{tr}}\tau_s,
\]
and
\[
\widetilde{B}_s=\Z/m_s\Z=\Z/m_{s,\mathrm{tr}}\Z\times \mathcal{G}_s,
\]
and $\mathcal L_{m_s}(\widetilde{B}_s)=1=\mathcal L_1(\mathcal{G}_s)$.
The remaining claim follows from Lemma~\ref{lem:L1-remove}.
Thus assume $m_{s,\mathrm{pow}}>1$.
Every prime divisor of $m_{s,\mathrm{pow}}$ satisfies $v_\ell(e^s-1)=0$, so the element $1-e^s$ is invertible modulo $m_{s,\mathrm{pow}}$; hence the fixed-point congruence
\[
\gamma_{\ast,s}\equiv e^s\gamma_{\ast,s}+1\pmod{m_{s,\mathrm{pow}}}
\]
has a unique solution in $(\Z/ m_{s,\mathrm{pow}}   \Z)^{\times}$.
The same analysis as in the proof of Lemma~\ref{lem:local-behavior} and Corollary~\ref{cor:no-mixed} then yields the decomposition
\[
\gamma_n  \bmod{m_s} \longmapsto \big(n\bmod m_{s,\mathrm{tr}},\ \gamma_{\ast,s}(1-(e^s)^n)\bmod m_{s,\mathrm{pow}}\big)
\]
under the Chinese Remainder Theorem isomorphism.
Exactly as in the proof of Theorem~\ref{thm:no-mixed-improvement}, the first coordinate has period $m_{s,\mathrm{tr}}$, the second has exact period $\tau_s$, and the hypothesis that $m_{s,\mathrm{tr}}$ and $\tau_s$ are co-prime implies that:
\[
L_s=\operatorname{lcm}(m_{s,\mathrm{tr}},\tau_s)=m_{s,\mathrm{tr}}\tau_s.
\]
This proves (i). Again by using the hypothesis that $m_{s,\mathrm{tr}}$ and $\tau_s$ are co-prime, the same Chinese Remainder Theorem argument as in the proof of Theorem~\ref{thm:no-mixed-improvement} gives the proof of (ii), namely that: 
\[
\widetilde{B}_s=\Z/m_{s,\mathrm{tr}}\Z\times \mathcal{G}_s.
\]

\bigskip

For (iii), compute the Fourier transform of $\mathbf 1_{\widetilde{B}_s}$ on the product group
\[
\Z/m_{s,\mathrm{tr}}\Z\times \Z/m_{s,\mathrm{pow}}\Z.
\]
As in the proof of Theorem~\ref{thm:no-mixed-improvement}, the full translation factor forces vanishing of all nontrivial frequencies in the first coordinate, and one obtains
\[
\mathcal L_{m_s}(\widetilde{B}_s)=\mathcal L_{m_{s,\mathrm{pow}}}(\mathcal{G}_s).
\]
Finally, Lemma~\ref{lem:L1-remove} gives
\[
\mathcal L_{m_s}(\widetilde B_{s,\mathrm{adm}})
\le \mathcal L_{m_s}(\widetilde{B}_s)+(s/\gcd(s,t))^{1/2}
=
\mathcal L_{m_{s,\mathrm{pow}}}(\mathcal{G}_s)+(s/\gcd(s,t))^{1/2}.
\]
\end{proof}

\begin{corollary}[Improved non-overlapping discrepancy bounds in the derived no-mixed regime]\label{cor:no-mixed-Dtilde}
Under the hypotheses of Theorem~\ref{thm:no-mixed-improvement-derived},
\begin{align}
D(\widetilde{\mathcal P}^{(s)})
&\le
1-\Big(1-\frac{1}{p^a}\Big)^{2rs}
+\frac{6q^{1/2}}{\widetilde N_s^{\ast}}\Big(\sum_{\iota=0}^{s-1}e^{2\iota}\Big)\big(\mathcal L_{m_{s,\mathrm{pow}}}(\mathcal{G}_s)+  \big( s/\gcd(s,t)\big)^{1/2}  \big)\big(C_0\ln(p^a)+1\big)^{2rs}\label{eq:no-mixed-Dtilde} 
\end{align}
and in particular since $\mathcal L_{m_{s,\mathrm{pow}}} (\mathcal{G}_s)\leq |\mathcal{G}_s|^{1/2} = \tau_s^{1/2}$ by Lemma~\ref{lem:L1-sqrt}, we have:
\begin{align}
D(\widetilde{\mathcal P}^{(s)})
&\le
1-\Big(1-\frac{1}{p^a}\Big)^{2rs}
+\frac{6q^{1/2}}{\widetilde N_s^{\ast}}\Big(\sum_{\iota=0}^{s-1}e^{2 \iota}\Big)\big(\tau_s^{1/2} +\big( s/\gcd(s,t)\big)^{1/2} \big)\big(C_0\ln(p^a)+1\big)^{2rs}\label{eq:no-mixed-Dtilde-rough}
\end{align}
\end{corollary}

\begin{proof}
Combine Theorem~\ref{thm:CaseC-Dtilde-L1}, Theorem~\ref{thm:no-mixed-improvement-derived}.
\end{proof}

\begin{theorem}[Pure power representation of the Fourier masses in the no-mixed regime]\label{thm:power-representation}
Under the hypotheses of Theorem~\ref{thm:no-mixed-improvement}, one has
\begin{equation}\label{eq:power-representation-G}
\mathcal L_{m_{\mathrm{pow}}}(\mathcal{G})
=
\frac{1}{m_{\mathrm{pow}}}
\sum_{b=0}^{m_{\mathrm{pow}}-1}
\left|
\sum_{v=0}^{\tau-1}\exp\!\Big(\frac{2\pi i\, b e^v}{m_{\mathrm{pow}}}\Big)
\right|.
\end{equation}
Under the hypotheses of Theorem~\ref{thm:no-mixed-improvement-derived}, one has
\begin{equation}\label{eq:power-representation-Gs}
\mathcal L_{m_{s,\mathrm{pow}}}(\mathcal{G}_s)
=
\frac{1}{m_{s,\mathrm{pow}}}
\sum_{b=0}^{m_{s,\mathrm{pow}}-1}
\left|
\sum_{v=0}^{\tau_s-1}\exp\!\Big(\frac{2\pi i\, b (e^s)^v}{m_{s,\mathrm{pow}}}\Big)
\right|.
\end{equation}
\end{theorem}

\begin{proof}
We prove \eqref{eq:power-representation-G}; the proof of \eqref{eq:power-representation-Gs} is identical after replacing
\[
(e,m_{\mathrm{pow}},\beta_{\ast},\tau,\mathcal{G})
\quad\text{by}\quad
(e^s,m_{s,\mathrm{pow}},\gamma_{\ast,s},\tau_s,\mathcal{G}_s).
\]

If $m_{\mathrm{pow}}=1$, then $\tau=1$ and $\mathcal{G}=\{0\}$, so both sides of \eqref{eq:power-representation-G} are equal to $1$. Thus assume $m_{\mathrm{pow}}>1$. By Theorem~\ref{thm:no-mixed-improvement},
\[
\mathcal{G}=\{\beta_{\ast}(1-e^v)\bmod m_{\mathrm{pow}}:0\le v<\tau\}.
\]
Therefore, for each $b$ modulo $m_{\mathrm{pow}}$,
\[
\sum_{g\in \mathcal{G}}\exp\!\Big(-\frac{2\pi i b g}{m_{\mathrm{pow}}}\Big)
=
\exp\!\Big(-\frac{2\pi i b\beta_{\ast}}{m_{\mathrm{pow}}}\Big)
\sum_{v=0}^{\tau-1}\exp\!\Big(\frac{2\pi i\, b\beta_{\ast} e^v}{m_{\mathrm{pow}}}\Big).
\]
Taking absolute values removes the phase factor $\exp(- 2 \pi i b \beta_{\ast}/m_{\mathrm{pow}}  )$.
Recall that $\beta_{\ast} \bmod{m_{\mathrm{pow}}}$ is invertible and so multiplication by $\beta_{\ast} \bmod{m_{\mathrm{pow}}}$ permutes the residue classes modulo $m_{\mathrm{pow}}$, so averaging over $b$ gives:
\[
\mathcal L_{m_{\mathrm{pow}}}(\mathcal{G})
=
\frac{1}{m_{\mathrm{pow}}}
\sum_{b=0}^{m_{\mathrm{pow}}-1}
\left|
\sum_{v=0}^{\tau-1}\exp\!\Big(\frac{2\pi i\, b e^v}{m_{\mathrm{pow}}}\Big)
\right|.
\]
\end{proof}

\begin{remark}
Theorem~\ref{thm:power-representation} shows that, in the no-mixed regimes, the remaining analytic input needed to further improve the estimates of the Fourier masses of $\mathcal{G}$ and $\mathcal{G}_s$, is good estimates for exponential sums along the multiplicative orbit
\[
1,e,e^2,\dots,e^{\tau-1}\pmod{m_{\mathrm{pow}}},
\]
(or more precisely, what is really needed is estimates of average of these sums with respect to the parameter $b$ as in the statement of Theorem~\ref{thm:power-representation}); similarly for the orbit generated by $e^s$ modulo $m_{s,\mathrm{pow}}$.
\end{remark}

As an illustration we show:

\begin{proposition}\label{prop:prime-power-component}
Assume the hypotheses of Theorem~\ref{thm:no-mixed-improvement}, and suppose in addition that $m_{\mathrm{pow}}=\ell^{\nu}$, where $\ell$ is an odd prime, and that $\tau=\ell^{\nu -1}(\ell-1)$. Then we have:
\[
\mathcal L_{m_{\mathrm{pow}}}(\mathcal{G}) =   (2\ell-3)/  \ell \leq2 .
\]

\bigskip
Consequently,
\begin{align}
D(\mathcal P)
&\le
1-\Big(1-\frac{1}{p^a}\Big)^{2r}
+\frac{12q^{1/2}}{N^{\ast}}
\big(C_0\ln(p^a)+1\big)^{2r},\label{eq:prime-power-comp-D}\\[0.5em]
D(\mathcal P^{(s)})
&\le
1-\Big(1-\frac{1}{p^a}\Big)^{2rs}
+\frac{6q^{1/2}}{N_s^{\ast}}
\Big(\sum_{\iota=0}^{s-1}e^{2\iota}\Big)
\left(
2+s^{1/2}
\right)\big(C_0\ln(p^a)+1\big)^{2rs}.\label{eq:prime-power-comp-Ds}
\end{align}
\end{proposition}

\begin{proof}
Since $m_{\mathrm{pow}}=\ell^{\nu}$ ($\ell$ is an odd prime) and $\tau=\ell^{\nu -1}(\ell-1)$, so $e$ is a primitive root $\bmod{\ell^{\nu}}$. Hence: 
\begin{eqnarray*}
 \quad\sum_{v=0}^{\tau-1}\exp\!\Big(\frac{2\pi i\, b e^v}{m_{\mathrm{pow}}}\Big) =\sum_{x \in (\Z/\ell^{\nu} \Z)^{\times}}\exp\!\Big(\frac{2\pi i\,  bx}{\ell^{\nu}}\Big) 
\end{eqnarray*}

\bigskip

For $1 \leq b \leq \tau-1 = \ell^{\nu-1}(\ell-1) -1$, write $b=c \ell^{f}$, where $f=v_{\ell}(b) \leq \nu-1$ and $\gcd(c,\ell)=1$. We then have:
\begin{eqnarray*}
 & &   \sum_{x \in (\Z/\ell^{\nu} \Z)^{\times}}\exp\!\Big(\frac{2\pi i\,  bx}{\ell^{\nu}}\Big)  \\
 &=&  \sum_{x \in (\Z/\ell^{\nu} \Z)^{\times}}\exp\!\Big(\frac{2\pi i\, cx}{\ell^{\nu -f}}\Big)  \\
 &=&  \sum_{x \in (\Z/\ell^{\nu } \Z)^{\times}}\exp\!\Big(\frac{2\pi i\,  x}{\ell^{\nu-f}}\Big)  \\
 &=&  \ell^{f} \sum_{x \in (\Z/\ell^{\nu -f} \Z)^{\times}}\exp\!\Big(\frac{2\pi i\,  x}{\ell^{\nu-f}}\Big) 
\end{eqnarray*}
(as each $x \in (\Z/\ell^{\nu -f} \Z)^{\times}$ has exaxctly $\ell^{f}$ pre-images in $x \in (\Z/\ell^{\nu } \Z)^{\times}$).

\bigskip
Now using the identity (the $\ell^n$-th cyclotomic polynomial, $n \geq 1$):
\begin{eqnarray*}
\frac{X^{\ell^{n}}-1}{X^{\ell^{n-1}}-1} &=& X^{\ell^{n-1}(\ell-1)} + X^{\ell^{n-1}(\ell-2)} \cdots + X^{\ell^{n-1}} + 1 \\
&=& \prod_{x \in (\Z/\ell^{n} \Z)^{\times}} \Big(X - \exp\!\Big(\frac{2\pi i\,  x}{\ell^{n}}\Big) \Big)
\end{eqnarray*}
we have that the value of:
\[
\sum_{x \in (\Z/\ell^{n} \Z)^{\times}}\exp\!\Big(\frac{2\pi i\,  x}{\ell^{n}}\Big) 
\]
is equal to $0$ if $n>1$, and is equal to $-1$ if $n=1$.

\bigskip
Hence for $1 \leq b \leq \tau-1$, the value of the sum 
\[
\sum_{x \in (\Z/\ell^{\nu} \Z)^{\times}}\exp\!\Big(\frac{2\pi i\,  bx}{\ell^{\nu}}\Big) 
\]
is equal to $0$ if $v_{\ell}(b) < \nu-1$, and is equal to $-\ell^{\nu-1}$ if $v_{\ell}(b)=\nu-1$. 

\bigskip

Now in the pure power representation of the Fourier mass $\mathcal{L}_{m_{\mathrm{pow}}}(\mathcal{G})=\mathcal{L}_{\ell^{\nu}}(\mathcal{G})$ in the no-mixed regime as given in Theorem~\ref{thm:power-representation}, for the outer sum over $0 \leq b \leq \tau-1$, the term $b=0$ contributes $\tau=\ell^{\nu -1}(\ell-1)$. While for terms corresponding to $b \neq 0$, it contributes the value $0$ if $v_{\ell}(b) < \nu-1$; while it contributes the value $\ell^{\nu-1}$ if $v_{\ell}(b) = \nu-1$, and there are $\ell-2$ such terms, namely $b=\ell^{\nu-1},2 \ell^{\nu-1},\cdots,(\ell-2) \ell^{\nu-1} $. Hence:
\begin{eqnarray*}
\mathcal{L}_{\ell^{\nu}}({\mathcal{G}}) &=& \frac{\ell^{\nu-1}(\ell-1) + \ell^{\nu-1}(\ell-2) }{\ell^{\nu}} \\
&=& \frac{2\ell-3}{\ell}
\end{eqnarray*}
In particular $\mathcal{L}_{\ell^{\nu}}({\mathcal{G}}) \leq 2$, and on applying Corollary~\ref{cor:no-mixed-discrepancy} finishes the proof.

\end{proof}

\subsection{Conclusion}
The results in Section~4 show that any bounds of the form
\begin{eqnarray}\label{eq:power-epsilon}
\frac{q^{1/2} \mathcal   L_m(B_{\mathrm{adm}})}{N^{\ast}} \ll  \frac{1}{q^{\epsilon}},\qquad
\frac{ q^{1/2}\mathcal L_m(B_{s,\mathrm{adm}}) }{N^{\ast}_s}\ll \frac{1}{q^{\epsilon}},\qquad
\frac{ q^{1/2}\mathcal L_{m_s}(\widetilde B_{s,\mathrm{adm}})}{\widetilde{N}^{\ast}_s}\ll \frac{1}{q^{\epsilon}}
\end{eqnarray}
(here $\epsilon >0$) would produce the corresponding discrepancy bounds for $\mathcal{P}$, $\mathcal{P}^{(s)}$, and $\widetilde{\mathcal{P}}^{(s)}$ respectively, and hence gives quantitative justification of the uniform distribution (respectively statistical independence) property of $\mathcal{P}$.

\bigskip
This suggests three concrete directions for further improvement. In the no-mixed regime, however, Section~5.1 already reduces the problem completely to the pure power components $\mathcal{G}$ and $\mathcal{G}_s$; in fact as we have seen in Corollary~\ref{cor:no-mixed-discrepancy} and Corollary~\ref{cor:no-mixed-Dtilde} in the no-mixed regime case, by using the elementary square root bounds $\mathcal L_m(\mathcal{G}) \leq |\mathcal{G}|^{1/2}$ and $\mathcal L_{m_{\mathrm{pow}}}(\mathcal{G}_s) \leq |\mathcal{G}_s|^{1/2}$, one can already give quite explicit sufficient conditions that would ensure that bounds of type \eqref{eq:power-epsilon} holds.  
\begin{itemize}[leftmargin=*]
\item interval estimates in the pure translation local regime;
\item exponential-sum estimates for multiplicative orbits in the pure power local regime; in the no-mixed regime, Theorem~\ref{thm:no-mixed-improvement}, Theorem~\ref{thm:no-mixed-improvement-derived}, and Theorem~\ref{thm:power-representation} reduce the relevant Fourier masses to explicit averages of such sums for the pure power components $\mathcal{G}$ and $\mathcal{G}_s$;
\item bilinear-sum methods over elliptic curves, following Shparlinski~\cite{Shparlinski2008Bilinear} and Ahmadi--Shparlinski~\cite{AhmadiShparlinski2010}, to treat the mixed regime directly.
\end{itemize}

\end{document}